\documentclass[11pt, a4paper, reqno]{amsart}
\usepackage[T1]{fontenc}
\usepackage[latin1]{inputenc}
\usepackage[english]{babel}
\usepackage[reqno]{amsmath}
\usepackage{amsthm}
\usepackage{latexsym}
\usepackage{amssymb}
\usepackage{mathrsfs}
\usepackage{mathtools}
\usepackage{mathabx}
\usepackage{lipsum}
\usepackage{titlesec}
\usepackage{tikz-cd}
\usepackage{dsfont} 

\usepackage[totalwidth=15cm,totalheight=20cm, hmarginratio=1:1]{geometry}
\usepackage{mathtools}
\usepackage{color}
\usepackage{physics}
\usepackage{faktor}
\usepackage[shortlabels]{enumitem}
\usepackage{tikz} 
\usepackage{hyperref}
\usepackage{graphicx}
\usepackage{tabularx}
\usepackage{cases} 
\newcolumntype{C}{>{\centering\arraybackslash}X} 
\usepackage{csquotes}
\usepackage{url}

\usepackage{cleveref}

\usepackage[style=alphabetic]{biblatex}

\usepackage{accents}

\makeatletter
\newtheorem*{rep@theorem}{\rep@title}
\newcommand{\newreptheorem}[2]{%
\newenvironment{rep#1}[1]{%
 \def\rep@title{#2 \ref{##1}}%
 \begin{rep@theorem}}%
 {\end{rep@theorem}}}
\makeatother

\makeatletter
\newtheorem*{rep@cor}{\rep@title}
\newcommand{\newrepcor}[2]{%
\newenvironment{rep#1}[1]{%
 \def\rep@title{#2 \ref{##1}}%
 \begin{rep@cor}}%
 {\end{rep@cor}}}
\makeatother

\makeatletter
\newtheorem*{rep@prop}{\rep@title}
\newcommand{\newrepprop}[2]{%
\newenvironment{rep#1}[1]{%
 \def\rep@title{#2 \ref{##1}}%
 \begin{rep@prop}}%
 {\end{rep@prop}}}
\makeatother

\newtheorem{theorem}{Theorem}[section]
\newreptheorem{theorem}{Theorem}
\numberwithin{theorem}{section}

\newtheorem{manualtheoreminner}{Theorem} 
\newenvironment{manualtheorem}[1]{%
    \renewcommand\themanualtheoreminner{#1}%
    \manualtheoreminner
}{\endmanualtheoreminner}

\newtheorem{lemma}[theorem]{Lemma}

\newrepcor{corollary}{Corollary}

\newtheorem{proposition}[theorem]{Proposition}

\newrepprop{prop}{Proposition}

\theoremstyle{definition}
\newtheorem*{definition*}{Definition}
\newtheorem{definition}[theorem]{Definition}

\theoremstyle{remark}
\newtheorem{remark}[theorem]{Remark}

\makeatletter
\def\paragraph{\@startsection{paragraph}{4}%
  \z@\z@{-\fontdimen2\font}%
  {\normalfont\bfseries}}
\makeatother

\numberwithin{equation}{section}
\allowdisplaybreaks

\patchcmd{\subsection}{-.5em}{.5em}{}{}

\makeatletter

\renewcommand\section{\@startsection{section}{1}%
  \z@{.7\linespacing\@plus\linespacing}{.5\linespacing}%
  {\normalfont\scshape\centering}}

\renewcommand\subsection{\@startsection{subsection}{2}%
  \z@{-.5\linespacing\@plus-.7\linespacing}{.5\linespacing}%
  {\bfseries}}
  
\renewcommand\subsubsection{\@startsection{subsubsection}{3}%
  \z@{-.5\linespacing\@plus-.7\linespacing}{.5\linespacing}%
  {\itshape}}
  
\def\l@paragraph{\@tocline{4}{0pt}{1pc}{7pc}{}}

\makeatother

\newcommand{\C}{\mathbb{C}}
\newcommand{\R}{\mathbb{R}}

\newcommand{\Z}{\mathbb{Z}}

\newcommand{\dsone}{\mathds{1}}

\newcommand{\Hyp}{\mathbb{H}^2}

\newcommand{\II}{I \! \! I}

\newcommand{\Lsl}{\mathfrak{sl}}

\newcommand{\diag}{\mathrm{diag}}

\newcommand{\vl}{|}

\newcommand{\T}{\mathcal{T}}

\newcommand{\Id}{\mathrm{Id}}

\newcommand{\Ree}{\mathcal{R}e}

\renewcommand{\i}{\mathbf{I}}

\newcommand{\g}{\mathbf{g}}

\newcommand{\Sg}{\Sigma}

\newcommand{\ome}{\boldsymbol{\omega}}

\newcommand{\quartic}{\mathcal Q^4\big(T^2\big)}
\newcommand{\almost}{\mathcal{J}(\R^2)}

\newcommand{\maximal}{\mathcal M(\R^2)}
\newcommand{\maximalflat}{\mathcal{M}\mathcal{F}(T^2)}

\newcommand{\dx}{\mathrm{d}x}
\newcommand{\dy}{\mathrm{d}y}
\newcommand{\du}{\mathrm{d}u}
\newcommand{\devu}{\mathrm{d}v}

\newcommand{\Xuno}{\mathbb{X}_{H_1}}
\newcommand{\Xdue}{\mathbb{X}_{H_2}}
\newcommand{\partialx}{\frac{\partial}{\partial x}}
\newcommand{\partialy}{\frac{\partial}{\partial y}}
\newcommand{\partialu}{\frac{\partial}{\partial u}}
\newcommand{\partialv}{\frac{\partial}{\partial v}}

\DeclarePairedDelimiterX{\scal}[2]{\langle}{\rangle}{#1 \mid #2}
\DeclarePairedDelimiterX{\scall}[2]{\langle}{\rangle}{#1, #2}

\DeclareMathOperator{\Imm}{Im}
\DeclareMathOperator{\Span}{Span}
\DeclareMathOperator{\Hom}{Hom}
\DeclareMathOperator{\End}{End}

\DeclareMathOperator{\SL}{\mathrm{SL}}

\DeclareMathOperator{\Ad}{Ad}

\begin{document}

\setcounter{secnumdepth}{3}
\setcounter{tocdepth}{2}

\title[Moduli space of flat maximal embeddings in pseudo-hyperbolic space]{The moduli space of flat maximal space-like embeddings in pseudo-hyperbolic space}

\author[Nicholas Rungi]{Nicholas Rungi}
\address{NR: Scuola Internazionale Superiore di Studi Avanzati (SISSA), Trieste (TS), Italy.} \email{nrungi@sissa.it} 

\author[Andrea Tamburelli]{Andrea Tamburelli}
\address{AT: Department of Mathematics, University of Pisa, Italy.} \email{andrea.tamburelli@libero.it}

\date{\today}

\begin{abstract}
We study the moduli space of flat maximal space-like embeddings in $\mathbb H^{2,2}$ from various aspects. We first describe the associated Codazzi tensors to the embedding in the general setting, and then, we introduce a family of pseudo-K\"ahler metrics on the moduli space. We show the existence of two Hamiltonian actions with associated moment maps and use them to find a geometric global Darboux frame for any symplectic form in the above family.
\end{abstract}

\maketitle

\tableofcontents

\section{Introduction}
This paper is the latest in a series that studies the pseudo-Riemannian and symplectic geometry of representations of surface groups into Lie groups of rank two (\cite{mazzoli2021parahyperkahler,rungi2021pseudo,rungi2023mathbb}). By work of Goldman (\cite{goldman1984symplectic}), character varieties of semi-simple Lie groups have a natural symplectic form. On the other hand, Labourie (\cite{labourie2017cyclic}) and Alessandrini-Collier (\cite{alessandrini2019geometry}) showed that in rank 2 many of their connected components enjoy a mapping class group invariant complex structure. The relation between the symplectic and complex geometry of character varieties appears, however, to be mysterious. The current paper has the purpose of shedding some light on the topic by introducing a pseudo-K\"ahler structure on the space of faithful representations of $\pi_1(T^2)$ into $\mathrm{SO}_0(2,3)$, up to conjugations. \\

First, given that $\mathrm{SO}_0(2,3)$ can be identified with the group of orientation-preserving isometries of $\mathbb{H}^{2,2}$, we build on the recent theory of maximal space-like surfaces developed in \cite{labourie2023quasicircles, LTW, Nie_alternating, TW} to identify the space of faithful representations of $\pi_1(T^2)$ into $\mathrm{SO}_0(2,3)$ with the space $\maximalflat$ of equivariant flat maximal embeddings of $\mathbb{R}^2$ into $\mathbb{H}^{2,2}$. Using the embedding data of such maximal surfaces, our first result gives a parameterization of $\maximalflat$ that is reminiscent of the analog parameterization of equivariant space-like maximal surfaces in $\mathbb{H}^{2,2}$ in higher genus (\cite{CTT}):

\begin{manualtheorem}A[Theorem \ref{thm:flatmaximalandquartic}] \label{thmA} The space $\maximalflat$ of equivariant flat maximal embeddings $\mathbb{R}^2 \hookrightarrow \mathbb{H}^{2,2}$ is parameterized by the complement of the zero section of the bundle $\mathcal{Q}^{4}(T^{2})$ of holomorphic quartic differentials over the Teichm\"uller space of the torus. 
\end{manualtheorem}

Next, we introduce new tensors related to maximal space-like embeddings into $\mathbb{H}^{2,2}$, which did not appear in the aforementioned papers, and use them to define a pseudo-K\"ahler structure on $\maximalflat$: 

\begin{manualtheorem}B[Theorem \ref{thm:pseudoKahler}] \label{thmB}
Let $f:[0,+\infty)\to(-\infty,0]$ be a smooth function such that:
\begin{itemize}
    \item[(i)]$f'(t)<0, \ \forall t\ge0$
    \item[(ii)] $f(0)=0$. 
\end{itemize}
Then the space $\maximalflat$ admits a $\mathrm{MCG}(T^2)$-invariant pseudo-K\"ahler structure $(\g_f,\ome_f,\mathbf{I})$ such that the complex structure $\mathbf{I}$ comes from the identification of $\maximalflat$ with the complement of the zero section in $\quartic$.\end{manualtheorem}

One of the main feature of this family of pseudo-Riemannian metrics, beside their compatibility with the natural complex structure on $\maximalflat$, is the fact that, once the function $f$ is chosen as in Theorem \ref{thmB}, they can be explicitly written in terms of a complex structure $J$ and the real part of a holomorphic quartic differential on the torus, thus potentially allowing further investigations on the geometric properties of these metrics. As an example, we study the behaviour of two natural actions on $\maximalflat$ with respect to these pseudo-K\"ahler structures. Indeed, identifying $\maximalflat$ with the complement of the zero section in $\quartic$, we see that $S^1$ acts on $\maximalflat$ by rotations along the fibers and $\mathbb{P}\mathrm{SL}(2,\R)$ acts on a pair $(J,q) \in\quartic$ by pull-back. It turns out these actions are geometric and can be used to better understand the symplectic forms $\ome_f$:

\begin{manualtheorem}C \label{thmC}
The $S^1$- and $\mathbb{P}\mathrm{SL}(2,\R)$-action on $\maximalflat$ are Hamiltonian with respect to $\ome_{f}$. 
Moreover, their Hamiltonians can be chosen as action variables in a global Darboux frame for $\ome_{f}$. 
\end{manualtheorem}

\subsection{Outline of the paper}
The paper is structured as follows: in Section \ref{subsec:background} we briefly introduce space-like maximal surfaces in $\mathbb{H}^{2,2}$ and describe their Gauss-Codazzi-Ricci equations. In Section \ref{subsec:tensors} we introduce new tensors that are closely related to holomorphic quartic differentials, and in Section \ref{subsec:torus} we analyze the equivariant flat case, showing that the associated moduli space can be parameterized as an open subset in the total space of a holomorphic bundle over the Teichm\"uller space of the torus. After introducing the proper tensors to be considered for the zero-curvature case and after studying their properties (Section \ref{sec:vector_bundle}), the main result of the paper, namely, the existence of the pseudo-K\"ahler metric, is contained in Section \ref{sec:pseudokahler}, the proof of which is also continued in Section \ref{sec:coordinatedescription}. Finally, other results regarding the properties of the pseudo-K\"ahler metric are obtained: the existence of an isometric and Hamiltonian circular action (Section \ref{sec:circleaction}), the existence of an $\SL(2,\R)$ Hamiltonian action (Section \ref{sec:SLaction}), and the existence of a geometric global Darboux frame for the symplectic form (Section \ref{sec:globalDarboux}).

\section{Maximal surfaces in \texorpdfstring{$\mathbb H^{2,2}$}{H2,2}}
In this section we recall some basic facts about maximal surfaces in the pseudo-hyperbolic space $\mathbb{H}^{2,2}$ and their geometric invariants. Most of the material covered here can be found in \cite{labourie2023quasicircles} and \cite{Nie_alternating}.

\subsection{Background} \label{subsec:background}

Let $\eta$ be a quadratic form in $\R^{5}$ of signature $(2,3)$. The pseudo-hyperbolic space is
\[
    \mathbb{H}^{2,2}=\mathbb{P}(\{x \in \R^{5} \ | \ \eta(x)<0\}) \ .
\]
The restriction of $\eta$ to the tangent space $T_{x}\mathbb{H}^{2,2}$ at each $x\in \mathbb{H}^{2,2}$ induces a pseudo-Riemannian metric $g$ of signature $(2,2)$ and of constant sectional curvature $-1$. Its topological boundary 
\[
    \partial_{\infty}\mathbb{H}^{2,2}=\mathbb{P}(\{x \in \R^{5} \ | \ \eta(x)=0\})
\]
is the boundary at infinity for the pseudo-Riemannian metric $g$ and it is endowed with a conformally flat Lorentzian structure.

\begin{definition} An embedded surface $\Sigma \subset \mathbb{H}^{2,2}$ is space-like if the restriction $g_{T}$ of $g$ to $T\Sigma$ is positive definite. We say that $g_{T}$ is the induced metric on $\Sigma$ and the surface is called \emph{complete} if its induced metric is complete.
\end{definition}
    
If $\Sigma \subset \mathbb{H}^{2,2}$ is a space-like surface, then its normal bundle $N\Sigma:=(T\Sigma)^{\perp_{g}}$ inside $T\mathbb{H}^{2,2}$ is endowed with a negative definite Riemannian metric $g_{N}$. Using the splitting $T\mathbb{H}^{2,2}=T\Sigma\perp_{g} N\Sigma$, the Levi-Civita connection on $\mathbb{H}^{2,2}$ decomposes as
\[
        \nabla=\begin{pmatrix}
            \nabla^{T} & B \\
            \II & \nabla^{N} 
        \end{pmatrix} \ ,
\]
where $\nabla^{T}$ is the Levi-Civita connection of the induced metric on $\Sigma$, $\II\in \Omega^{1}(\Sigma, \Hom(T\Sigma, N\Sigma)$ is the \emph{second fundamental form}, $B\in \Omega^{1}(\Sigma, \Hom(N\Sigma, T\Sigma))$ is the \emph{shape operator}, and $\nabla^{N}$ is compatible with the metric $g_{N}$. Since $\nabla$ is torsion-free, the second fundamental form is actually symmetric, in the sense that
\[
    \II(X,Y)=\II(Y,X) \ \ \ \ \forall X,Y\in \Gamma(T\Sigma)  .
\]
Moreover, the second fundamental form and the shape operator are related by
\[
    g_{N}(\II(X,Y), \xi) = -g_{T}(Y,B(X,\xi))
\]
for all $X, Y \in \Gamma(T\Sigma)$ and normal vector fields $\xi$. 

\begin{definition} A space-like surface $\Sigma \subset \mathbb{H}^{2,2}$ is maximal if $\trace_{g_{T}}(\II)=0$.
\end{definition}

\noindent It follows easily from the definition that a surface $\Sigma$ is maximal if and only if in a $g_{T}$-orthonormal frame $\{e_{1}, e_{2}\}$ of $T\Sigma$, we have 
\[
    \II(e_{1},e_{1})=-\II(e_{2},e_{2}) \ .
\]

\noindent The induced metric, the second fundamental form and the metric on the normal bundle of a maximal surface satisfy some fundamental equations:
\begin{enumerate}[i)]
    \item the Gauss equation
    \[
        K_{g_{T}} = -1+\frac{1}{2}\|\II\|^{2} \ ,
    \]
    where $K_{g_{T}}$ denotes the Gaussian curvature of the induced metric $g_{T}$ on $\Sigma$ and the norm of the second fundamental form is defined by the formula
    \[
        \| \II \|^{2} = -g_{N}(\II(e_{1},e_{1}), \II(e_{1},e_{1})) -g_{N}(\II(e_{2},e_{1}), \II(e_{2},e_{1}))
    \]
    for a $g_{T}$-orthonormal frame $\{e_{1}, e_{2}\}$ of $T\Sigma$;
    \item the Ricci equation
    \[
        R^{N}=\omega \otimes \varphi \ ,
    \]
    where $R^{N}$ denoted the curvature tensor of $g_{N}$, $\omega$ is the area form on $\Sigma$ for its induced metric, and $\varphi$ is the unique endomorphism of $N\Sg$ such that
    \[
        g_{T}(B(X,\xi), B(Y,\eta))-g_{T}(B(Y,\xi), B(X,\eta))=\omega(X,Y)g_{N}(\varphi(\xi),\eta)
    \]
    for all $X,Y\in \Gamma(T\Sigma)$ and $\xi,\eta\in \Gamma(N\Sigma)$.
    \item the Codazzi equation 
    \[
            d^{D}\II=0
    \]
    where $D$ is the connection induced on $\Hom(T\Sigma, N\Sigma)$ by $\nabla$ and $\nabla^{N}$ and the operator $d^{D}$ is defined by
    \[
        (d^{D}\theta)(X)Y:= D_{X}(\theta(Y))-D_{Y}(\theta(X))-\theta([X,Y])
    \] 
    for $X,Y\in \Gamma(T\Sigma)$. 
\end{enumerate}

\noindent As a consequence of the Codazzi equation, it is possible to introduce some holomorphic data on a maximal surface $\Sigma \subset \mathbb{H}^{2,2}$. Let $J$ be the complex structure on $\Sigma$ compatible with its induced metric and denote by $T^{\C}\Sigma$ and $N^{\C}\Sigma$ the complexifications of the tangent and normal bundle of $\Sigma$ with respect to this complex structure. Let $K_{\Sigma}$ be the holomorphic cotangent bundle of $\Sigma$. By \cite[Lemma 5.6]{labourie2023quasicircles}, there exists a holomorphic section $\sigma$ of $K^{2}_\Sg \otimes N^{\mathbb C}\Sigma$ such that $\II=\Ree(\sigma)$. Moreover, the Gauss equation can be rewritten in terms of $\sigma$ as
\[
        K_{g_{T}}=-1+\|\sigma\|^{2}_{h} \ , 
\]
where $h$ is the Hermitian extension of $g_{T}^{2}\otimes (-g_{N})$ on $K^{2}_\Sg\otimes N^{\C}\Sigma$.

\noindent The main holomorphic object that we are going to use in the sequel is the holomorphic quartic differential $q$ on $\Sigma$ defined out of $\sigma$ by the expression
\[
        q:=-\frac{1}{4}g_{N}(\sigma, \sigma) \ . 
\]

\subsection{Other tensors}\label{subsec:tensors}

In this section we define other tensors related to maximal surfaces in $\mathbb{H}^{2,2}$.

\begin{theorem} \label{thm:tensorsandquarticdifferentials}
Let $S$ be a smooth oriented surface endowed with a Riemannian metric $h$ and let $J$ be the (almost) complex structure defined by the conformal class of $h$. Suppose that a $(0,4)$ tensor $T$ and a $(1,3)$ tensor $U$ are related by $U=h^{-1}T$, then 
\begin{enumerate}[(i)]
    \item $T$ is totally symmetric if and only if $U(X,Y)=U(Y,X)$ is an $h$-symmetric endomorphism such that $U(X,Y)Z=U(Z,X)Y=U(Y,Z)X$ for any $X,Y,Z\in\Gamma(TS)$;
    \item $T$ is the real part of a complex quartic differential $q$ on $(S,J)$ if and only if the endomorphism $U(X,Y)$ is trace-less and $U(JX,Y)=U(X,JY)=U(X,Y)J$ for any $X,Y\in\Gamma(TS)$; 
    \item $q$ is holomorphic if and only if $\mathrm d^{\nabla} T=0$, where $\nabla$ is the Levi-Civita connection of $h$. Moreover, the quartic differential can be written as $$q=T(\cdot,\cdot,\cdot,\cdot)-iT(\cdot,\cdot,\cdot,J\cdot) \ .$$
\end{enumerate} 
\end{theorem}
\begin{proof} By definition, the tensor $U$ is characterized by the property that
\[
        T(X,Y,Z,W)=h(U(X,Y)Z,W)  \ \ \ \  \forall \ X,Y,Z,W \in \Gamma(TS) \ .
\]
\begin{enumerate}[(i)]
    \item Let us prove that $U(X,Y)$ is an $h$-symmetric endomorphism:
        \begin{align*}
            h(U(X,Y)Z,W)&=T(X,Y,Z,W)  \tag{$T$ is totally symmetric} \\
                        &=T(X,Y,W,Z)  \\
                        &=h(U(X,Y)W,Z) \ .
        \end{align*}
    The proofs of the other properties are similar and left to the reader.
    \item Define $q=T(\cdot, \cdot, \cdot, \cdot)-iT(\cdot, \cdot, \cdot, J\cdot)$: it is a complex quartic differential with real part $T$ if and only if $q(J\cdot, \cdot, \cdot, \cdot)=q(\cdot, J\cdot, \cdot, \cdot)=q(\cdot, \cdot, J\cdot, \cdot)=q(\cdot, \cdot, \cdot, J\cdot)=iq(\cdot, \cdot, \cdot, \cdot).$ Since the last equality clearly holds, it is sufficient to prove that $T(J\cdot, \cdot, \cdot, \cdot)=T(\cdot, J\cdot, \cdot, \cdot)=T(\cdot, \cdot, J\cdot, \cdot)=T(\cdot, \cdot, \cdot, J\cdot)$. In terms of the tensor $U$, these conditions are equivalent to $U(JX,Y)=U(X,JY)=U(X,Y)J=-JU(X,Y)$ for all $X,Y \in \Gamma(TS)$. Notice that the last equality says that the $h$-symmetric endomorphism $U(X,Y)$ anti-commutes with $J$ and this can happen if and only if $U(X,Y)$ is trace-less. 
    \item We do this computation in a local holomorphic normal coordinate $z=x+iy$ on $S$, which exists because the metric $h$ is trivially K\"ahler. By item $(ii)$, the quartic differential $q$ can be written locally as $q=fdz^{4}$ for some smooth function $f$. We need to prove that $f$ is holomorphic if and only if the tensor $T$ is Codazzi. Now, using the symmetries of the tensor $T$ proved in part $(i)$ and $(ii)$, the Codazzi equation $0=(d^{\nabla}T)(\partial_{x},\partial_{y},\cdot,\cdot,\cdot):=(\nabla_{\partial_{x}}T)(\partial_{y},\cdot,\cdot,\cdot)-(\nabla_{\partial_{y}}T)(\partial_{x},\cdot,\cdot,\cdot)$ reduces to
    \[
        \begin{cases}
            \partial_{x}(T(\partial_{y}, \partial_{x},\partial_{x},\partial_{x})=\partial_{y}(T(\partial_{x}, \partial_{x},\partial_{x},\partial_{x}) \\
             \partial_{x}(T(\partial_{x}, \partial_{x},\partial_{x},\partial_{x})=-\partial_{y}(T(\partial_{y}, \partial_{x},\partial_{x},\partial_{x})
        \end{cases} \ .
    \]
    Since $f=q(\partial_{z}, \partial_{z}, \partial_{z}, \partial_{z})=T(\partial_{x},\partial_{x},\partial_{x},\partial_{x})-iT(\partial_{x},\partial_{x},\partial_{x},\partial_{y}))$, we see that the above equations are exactly the Cauchy-Riemann equations for $f$, hence $q$ is holomorphic.
\end{enumerate}
\end{proof}

\noindent As a consequence, if $\Sigma$ is a maximal surface in $\mathbb{H}^{2,2}$ with associated quartic differential $q$, the real part $T$ of $q$ is a $(0,4)$-tensor satisfying the properties of Theorem \ref{thm:tensorsandquarticdifferentials}. It follows from the definitions that $T$ can be expressed in terms of the second fundamental form as  
\[
    T= g_{N}(\II(\cdot, \cdot), \II(\cdot, \cdot))-g_{N}(\II(\cdot, J\cdot), \II(\cdot, J\cdot)) \ .
\]
\subsection{The case of the torus}\label{subsec:torus}

In this paper, we are interested in complete maximal embeddings $\mathbb{R}^{2} \rightarrow \mathbb{H}^{2,2}$ that are equivariant with respect to a faithful representation $\rho:\pi_{1}(T^{2}) \rightarrow \mathrm{SO}_{0}(2,3)$. In this setting the tensors defined in Section \ref{subsec:background} and \ref{subsec:tensors} simplify considerably. First, by \cite[Corollary 5.7]{labourie2023quasicircles} a complete maximal surface in $\mathbb{H}^{2,2}$ is negatively curved, so an application of Gauss-Bonnet implies that it must be flat. Moreover, the quartic differential $q$ is constant and non-zero because, otherwise, since $q=-\frac{1}{4}g_{N}(\sigma, \sigma)$, the holomorphic section $\sigma$ and hence $\II=\Ree(\sigma)$ would vanish identically, implying that the maximal surface would be totally geodesic, but in $\mathbb{H}^{2,2}$ totally geodesic space-like surfaces are isometric to $\mathbb{H}^{2}$. Complete flat maximal surfaces in $\mathbb{H}^{2,2}$ are well-known: they are all obtained as orbits of a point under the action of a Cartan subgroup of $\mathrm{SO}_{0}(2,3)$ and are known in the literature as Barbot surfaces (\cite{labourie2023quasicircles}) or horospherical surfaces (\cite{BS_maximal,Tambu_poly, TW}). In particular, their geometry is completely determined by a conformal structure and a non-zero holomorphic quartic differential:

\begin{theorem}\label{thm:flatmaximalandquartic} The space of complete equivariant flat maximal surfaces $(\Sigma, \rho)$ up to the action of $\mathrm{SO}_{0}(2,3)$, is parameterized by the bundle of non-zero holomorphic quartic differentials over the Teichm\"uller space of the torus $\T(T^{2})$.
\end{theorem}
\begin{proof} Let $\rho:\pi_{1}(T^{2}) \rightarrow \mathrm{SO}_{0}(2,3)$ be a faithful representation and let $\Sigma$ be a $\rho$-equivariant maximal surface in $\mathbb{H}^{2,2}$. Since all complete flat maximal surfaces in $\mathbb{H}^{2,2}$ are equivalent under the action of $\mathrm{SO}_{0}(2,3)$ (\cite[Proposition 2.27]{labourie2023quasicircles}), up to conjugation, we can assume that $\Sigma$ is the standard Barbot surface $\Sigma_{0}$. An explicit parameterization of $\Sigma_{0}\subset \mathbb{H}^{2,2}$ is given by 
\[
    f_{0}(x,y)=\frac{1}{2}\left( \sqrt{2}\sinh(2y), \sqrt{2}\sinh(2x), \cosh(2x)+\cosh(2y), \cosh(2x)-\cosh(2y),0 \right) \ .
\]
Note that $\Sigma_{0}$ intersects the boundary at infinity of $\mathbb{H}^{2,2}$ in four points $[0,\pm\sqrt{2},1,1,0]$ and $[\pm\sqrt{2},0, 1,-1,0]$ connected by light-like segments. 
Since $\Sigma_{0}$ endowed with the complex structure compatible with its induced metric is biholomorphic to $\C$, the quotient $\Sigma_{0}/\rho(\pi_{1}(T^{2}))$ defines a point $J\in \T(T^{2})$. Let $q$ be the holomorphic quartic differential on the torus induced by $\Sigma_{0}$. The map that associates to $(\Sigma,\rho)$ the pair $(J,q)$ constructed in this way is a bijection. Indeed, it is easy to check that the parameterization $f_{0}$ is conformal and the associated quartic differential is $q=dz^{4}$, for $z=x+iy$, thus all pairs $(J,q)$ can be obtained by pre-composing $f_{0}$ with an element of $\mathrm{GL}(2,\mathbb{R})$. On the other hand, let $(\Sigma_{1},\rho_{1})$ and $(\Sigma_{2},\rho_{2})$ be two complete equivariant flat maximal surfaces with the same data $(J,q)$. After renormalizing the surfaces by post-composing with a global isometry of $\mathbb{H}^{2,2}$, we can assume that they coincide with the standard Barbot surface $\Sigma_{0}$. Let $f_{i}:\mathbb{C} \rightarrow \Sigma_{0}$ be conformal $\rho_{i}$-equivariant parameterizations of $\Sigma_{0}$. Then $\phi=f_{1}^{-1}\circ f_{2}$ is a biholomorphism of $\C$ such that $\phi^{*}q=q$. The only possibilities are $\phi(z)=\zeta_{4}z$ for some forth root of unity $\zeta_{4}$. However, these coordinate changes have the effect of cyclically moving the four points at infinity of $\Sigma_{0}$, which can also realized by post-composing with an isometry $A\in \mathrm{SO}_{0}(2,3)$ preserving the Barbot surface. Therefore, we can conclude that $f_{1}=A\circ f_{2}$ and $\rho_{1}$ is conjugate to $\rho_{2}$.
\end{proof}
\begin{remark}
The parameterization $f_0$ given in the previous theorem is actually contained in the pseudo-hyperbolic space $\mathbb{H}^{2,1}$, which coincides with the $3$-dimensional Anti-de Sitter space. We thus find the same surface studied by Bonsante-Seppi (\cite{bonsante2020anti}) in the case of equivariant flat immersions in $\mathbb{H}^{2,1}$. 
\end{remark}
\section{Geometry of the moduli space}
In this section we focus on the case $\Sg=T^2$, hence we are going to consider tensors associated with flat maximal space-like embeddings. For this reason, the complex structure and the second fundamental form induced on the surface, as well as the metric on the normal bundle, will no longer depend on the point but will be constant tensors defined on $\R^2$. A similar approach was first used in \cite[\S 3.2,\S 3.3]{mazzoli2021parahyperkahler}, and then in \cite[\S 3.1]{rungi2021pseudo}.

\subsection{Linear almost-complex structures}\label{sec:linearalmost}
Let $\rho_0:=\mathrm{d}x_0\wedge\mathrm{d}y_0$ be the standard area form on $\R^2$ and let us introduce $$\almost:=\{J\in\End(\R^2) \ | \ J^2=-\mathds{1}, \ \rho_0(v,Jv)>0 \ \text{for some} \ v\in\R^2\setminus\{0\}\} $$ to be the space of $\rho_0$-compatible linear-complex structures on $\R^2$. It is a real two dimensional manifold and for any $J\in\almost$, the pairing $g_J(\cdot,\cdot):=\rho_0(\cdot,J\cdot)$ gives a well-defined scalar product on $\R^2$, with respect to which $J$ is an orthogonal endomorphism. By differentiating the identity $J^2=-\mathds{1}$, it follows that $$T_J\almost=\{\dot{J}\in\End(\R^2) \ | \ J\dot J+\dot J J=0\} \ . $$ Equivalently, the space $T_J\almost$ can be identified with the trace-less and $g_J$-symmetric endomorphisms of $\R^2$. It carries a natural (almost) complex structure given by
\begin{align*}
\mathcal{I}:T_J&\almost\to T_J\almost \\ &\dot J\mapsto-J\dot J \ . 
\end{align*}
There is a natural scalar product defined on the tangent space 
$$\langle\dot J, \dot J'\rangle_J:=\frac 1{2}\tr(\dot J\dot J')$$for each $\dot J,\dot J'\in\almost$ and it is easy to check that $\mathcal{I}$ preserves $\langle\cdot,\cdot\rangle_J$. It is also possible to define a symplectic form on $T_J\almost$ by $$\Omega_J(\dot J,\dot J')=-\langle \dot J,J\dot J'\rangle_J \ .$$ The group $\SL(2,\R)$ acts by conjugation on $\almost$: for any $A\in\SL(2,\R)$ one defines $A\cdot J:=AJA^{-1}$. The same formula can be used to define an action on $T_J\almost$. \begin{lemma}\label{lem:kahleronalmost}
The triple $\big(\langle\cdot,\cdot\rangle,\mathcal I,\Omega\big)$ defines a $\SL(2,\R)$-invariant K\"ahler metric on $\almost$.
\end{lemma}
\begin{lemma}[{\cite[Lemma 4.3.2]{trautwein2018infinite}}]\label{lem:Trautweinkahlerhyperbolic}
Let $\Hyp$ be the hyperbolic plane with complex coordinate $z=x+iy$ and with K\"ahler structure $$g_{\Hyp}=\frac{\mathrm{d}x^2+\mathrm{d}y^2}{y^2}, \qquad\qquad \omega_{\Hyp}=-\frac{\mathrm{d}x\wedge\mathrm{d}y}{y^2} \ . $$ Then, there exists a unique $\SL(2,\R)$-invariant K\"ahler isometry $j:\Hyp\to\almost$ such that $j(i)=J_0=\begin{psmallmatrix}
    0 & -1 \\ 1 & 0\end{psmallmatrix}$. It is given by the formula \begin{equation}\label{eq:kahlerisometryalmost}
    j(x+iy):=\begin{pmatrix}
    \frac{x}{y} & -\frac{x^2+y^2}{y} \\ \frac 1{y} & -\frac x{y} 
    \end{pmatrix} \ . 
\end{equation}
\end{lemma}
After defining Teichm\"uller space of the torus $\mathcal T(T^2)$ as the set of isotopy classes of unit area flat metrics on $T^2$, then it is easy to see that there is a $\mathrm{MCG}(T^2)\cong\SL(2,\Z)$-equivariant isomorphism between $\mathcal T(T^2)$ and $\almost$. In particular, the K\"ahler metric $\big(\langle\cdot,\cdot\rangle,\mathcal I,\Omega\big)$ can be pulled-back on $\mathcal T(T^2)$ using the equivariant isomorphisms described above.\newline We conclude with a technical result that will be useful later on. \begin{lemma}[{\cite[Lemma 3.9]{mazzoli2021parahyperkahler}}]\label{lem:dotJ}
For every $\dot J,\dot J',\dot J''\in T_J\almost$ the following holds \newline $(i)$ $\dot J\dot J'=\langle\dot J,\dot J'\rangle\mathds{1}-\langle J\dot J,\dot J'\rangle J$;
\newline $(ii)$ $\tr(\dot J\dot J'\dot J'')=\tr(J\dot J\dot J'\dot J'')=0 \ .$
\end{lemma}

\subsection{A vector bundle over Teichm\"uller space} \label{sec:vector_bundle}
The next step is to define a vector bundle over $\almost$, using the pairwise description $(J,U)$ where $U$ is the tensor defined in Section \ref{subsec:tensors}, so that the fiber over a point in Teichm\"uller space $\mathcal T(T^2)$ can be identified with the space of holomorphic quartic differentials on $(T^2,J)$. In this regard, let us consider the following set 
\begin{equation}\label{eq:vectorbundleoveralmost}
    \maximal:=\{(J,T)\in\almost\times S_4(\R^2) \ | \ T(J\cdot,J\cdot,\cdot,\cdot)=-T\}
\end{equation}
where $S_4(\R^3)$ is the space of totally-symmetric $(0,4)$ tensors. It is a smooth manifold of real dimension $4$ which fibre over $\almost$. Moreover, for any $(J,T)\in\maximal$ we can define the associated $(1,3)$ tensor $U:=g_J^{-1}T$, namely $U$ is completely determined by the formula \begin{equation}g_J\big(U(X,Y)Z,W\big)=T(X,Y,Z,W) , \ \text{for any} \ X,Y,Z,W\in\R^2 \ .\end{equation} Since $T$ is totally-symmetric, we have that $U(X,Y)Z=U(Z,X)Y=U(Y,Z)X$ for any $X,Y,Z\in\R^2$. Hence, we can consider the element $U$ as a bi-linear symmetric form in the first two entries so that for any vectors $X,Y\in\R^2$, we obtain an endomorphism $U(X,Y)$ of $\R^2$ which is symmetric with respect to $g_J$. \begin{lemma}\label{lem:propertiesUeT}
For any point $(J,T)$ in $\maximal$ the following properties hold: \begin{enumerate}
\item[(i)] $T(J\cdot,J\cdot,J\cdot,J\cdot)=T$;
    \item[(ii)] $T(J\cdot,\cdot,\cdot,\cdot)=T(\cdot,J\cdot,\cdot,\cdot)=T(\cdot,\cdot,J\cdot,\cdot)=T(\cdot,\cdot,\cdot,J\cdot);$ \item[(iii)] $U(JX,Y)Z=U(X,JY)Z=U(X,Y)JZ , \ \text{for any} \ X,Y,Z\in\R^2;$ \item[(iv)] $JU(X,Y)Z=-U(X,Y)JZ , \ \text{for any} \ X,Y,Z\in\R^2;$ \item[(v)] The endomorphism part of $U$ is trace-less.  
\end{enumerate}
\end{lemma}\begin{proof}
Throughout the proof we will denote by $X,Y,Z$ and $W$ arbitrary vectors in $\R^2$. \newline $(i)$ The first property can be easily deduced: \begin{align*}T(JX,JY,JZ,JW)&=-T(J^2X,J^2Y,JZ,JW) \tag{$T\in\maximal$} \\ &=-T(JZ,JW,X,Y) \tag{$T$ \text{is totally-symmetric}} \\ &=T(J^2Z,J^2,X,Y) \tag{$T\in\maximal$} \\ &=T(X,Y,Z,W) \ . \tag{$T$ \text{is totally-symmetric}}\end{align*} $(ii)$ The second claim follows from (\ref{eq:vectorbundleoveralmost}) and the following computation: $$T(JX,Y,Z,W)=T(J^2X,JY,JZ,JW)=T(X,JY,J^2Z,J^2W)=T(X,JY,Z,W) \ ,$$ and the same applies to the other entries. \newline $(iii)$ The third one can be deduced from the definition of $U$ and an application of $(ii)$. In fact, \begin{align*}
    g_J\big(U(JX,Y)Z,W\big)&=T(JX,Y,Z,W) \\ &=T(X,JY,Z,W) \\ &=g_J\big(U(X,JY)Z,W\big) \\ &= T(X,Y,JZ,W) \\ &=g_J\big(U(X,Y)JZ,W\big) \ .
\end{align*}$(iv)$ The fourth property is a consequence of the relation $g_J(J\cdot,J\cdot)=g_J$, as shown by the calculations performed below \begin{align*}
    g_J\big(JU(X,Y)Z,W\big)&=g_J\big(J^2U(X,Y)Z,JW\big) \\ &=-T(X,Y,Z,JW) \\ &=-T(X,Y,JZ,W) \tag{Property $(ii)$} \\ &=-g_J\big(U(X,Y)JZ,W\big) \ .
\end{align*}$(v)$ Since the endomorphism $U(X,Y)$ anti-commute with $J$ it is both $g_J$-symmetric and trace-less.
\end{proof}
\begin{proposition}\label{prop:isomorphismmaximalandquartic}
There is an isomorphism as smooth vector bundles between $\quartic$ and $\maximal$, where $\quartic$ is the vector bundle of holomorphic quartic differentials over $\mathcal{T}(T^2)$.
\end{proposition}\begin{proof}
Recall that $\maximal$ has the structure of a real vector bundle over $\almost$ and $\quartic$ is a holomorphic vector bundle over $\mathcal T(T^2)$ endowed with the complex structure described in Section \ref{sec:linearalmost}. To any point $(J,T)\in\maximal$ we can associate the $(1,3)$ tensor $U:=g_J^{-1}T$ satisfying the properties of Lemma \ref{lem:propertiesUeT}. In particular, the endomorphism part of $U$ is $g_J$-symmetric and trace-less, thus by Theorem \ref{thm:tensorsandquarticdifferentials} there exists a holomorphic quadratic differential $q$ whose real part is equal to $T$ and which is uniquely determined by the pair $(J,T)$. In particular, we obtain a well-defined map \begin{align*}\Phi: \ &\mathcal{M}(\R^2)\longrightarrow\quartic \\ &(J,T)\longmapsto (J,q)\end{align*}which is easily shown to be the natural lifting of the isomorphism $\almost\cong\mathcal{T}(T^2)$, at the level of the corresponding smooth vector bundles. 
\end{proof}In the following we will go on to explore further properties of the $U$ tensor or, equivalently, those of $T$. Recall that for each $(J,U)\in\maximal$ we can think of $U$ as a symmetric bi-linear form in the first two entries so that for each $X,Y\in\R^2$ the element $U(X,Y)$ is a $g_J$-symmetric and trace-less endomorphism of $\R^2$ (Lemma \ref{lem:propertiesUeT}). Let $\{e_1,e_2\}$ be a $g_J$-orthonormal basis of $\R^2$ such that $Je_1=e_2$ and with dual basis denoted by $\{e_1^*,e_2^*\}$. The set of all trace-less and symmetric bi-linear forms on $\R^2$ with respect to $g_J$ is a real vector space of dimension two with a basis given by \begin{equation}\label{eq:basissymmetricbilinear}\{e_1^*\otimes e_1^*-e_2^*\otimes e_2^*; e_1^*\otimes e_2^*+e_2^*\otimes e_1^*\} \ .\end{equation} Thus, we can write $$U=U_1\otimes\big(e_1^*\otimes e_1^*-e_2^*\otimes e_2^*\big)+U_2\otimes\big(e_1^*\otimes e_2^*+e_2^*\otimes e_1^*\big) \ , $$ where $U_1:=U(e_1,e_1)$ and $U_2:=U(e_1,e_2)$ are trace-less and $g_J$-symmetric matrices.
\begin{lemma}\label{lem:JUunoUdue}
For any $(J,U)\in\maximal$ and using the notations above, we get $JU_2=U_1$.
\end{lemma}
\begin{proof}
Let us write the matrices $U_1, U_2$ and $JU_2$ in the $g_J$-orthonormal basis as $$U_1=\begin{pmatrix}
    a_1 & b_1 \\ c_1 & d_1
\end{pmatrix}, \quad U_2=\begin{pmatrix}
    a_2 & b_2 \\ c_2 & d_2
\end{pmatrix}, \quad JU_2=\begin{pmatrix}
    -c_2 & -d_2 \\ a_2 & b_2
\end{pmatrix} \ . $$ We prove that the four entries of $U_1$ are the same as those of $JU_2$. Details will be given for only one case, but the other three follow easily. In fact, \begin{align*}
    a_1=g_J\big(U_1\cdot e_1,e_1\big)=g_J\big(U(e_1,e_1)e_1,e_1\big)=T(e_1,e_1,e_1,e_1)
\end{align*}and exploiting the properties of Lemma \ref{lem:propertiesUeT} we conclude \begin{align*}
    -c_2&=g_J\big(JU_2\cdot e_1,e_1\big)=g_J\big(JU(e_1,e_2)e_1,e_1\big)=-g_J\big(U(e_1,e_2)e_2,e_1\big) \\ &=-T(e_1,Je_1,Je_1,e_1)=T(e_1,e_1,e_1,e_1) \ .
\end{align*}
\end{proof}
 Let $(J,U)$ and $(J',U')$ be two elements in $\maximal$ and let us denote with $\sigma_1,\sigma_2$ the first and second element of the basis in (\ref{eq:basissymmetricbilinear}), respectively. Then, if we write $U=U_1\otimes\sigma_1+U_2\otimes\sigma_2$ and $U'=U_1'\otimes\sigma_1+U_2'\otimes\sigma_2$, we can introduce a scalar product \begin{equation}\label{eq:scalarproductU}\begin{aligned}
    \langle U,U'\rangle_J:&=\frac{1}{4}\Big(\tr(U_1U_1')\tr(g_J^{-1}\sigma_1g_J^{-1}\sigma_1)+\tr(U_2U'_2)\tr(g_J^{-1}\sigma_2g_J^{-1}\sigma_2)\Big) \\ &=\frac{1}{2}\Big(\tr(U_1U_1')+\tr(U_2U_2')\Big) \ .
\end{aligned}\end{equation}In particular, the norm of a tensor $U$ can be expressed as $$\vl\vl U\vl\vl_J^2=\frac{1}{2}\Big(\tr((U_1)^2)+\tr((U_2)^2)\Big)=\tr((U_1)^2) \ ,$$ since $-U_2=JU_1$ according to Lemma \ref{lem:JUunoUdue}. We define an $\SL(2,\R)$-action on such tensors by the following formula \begin{equation}\label{eq:SLactiononU}
    A\cdot U:=AU(A^{-1}\cdot,A^{-1}\cdot)A^{-1}, \quad A\in\SL(2,\R) \ .
\end{equation}\begin{lemma}
    The scalar product $\langle\cdot,\cdot\rangle_J$ defined in (\ref{eq:scalarproductU}) is invariant by the $\SL(2,\R)$-action.
\end{lemma}
\begin{proof}
Let $U$ and $U'$ be in the fibre of $\maximal\to\almost$ over a point $J$, and let $A$ be a matrix in $\SL(2,\R)$, then we need to prove that $$\langle A\cdot U, A\cdot U'\rangle_{AJA^{-1}}=\langle U,U'\rangle_J \ .$$ First notice that if we write $U=U_1\otimes\sigma_1+U_2\otimes\sigma_2$ and $U'=U_1'\otimes\sigma_1+U_2'\otimes\sigma_2$, then $\SL(2,\R)$-action can be expressed as \begin{align*}
    &A\cdot U=AU_1A^{-1}\otimes(A^{-1})^*\sigma_1+AU_2A^{-1}\otimes(A^{-1})^*\sigma_2 \\ &A\cdot U'=AU_1'A^{-1}\otimes(A^{-1})^*\sigma_1+AU_2'A^{-1}\otimes(A^{-1})^*\sigma_2 \ .
\end{align*} By \cite[Lemma 3.3]{mazzoli2021parahyperkahler} we know that $g_{AJA^{-1}}^{-1}(A^{-1})^*\sigma_i=A(g_J^{-1}\sigma_i)A^{-1}$ for $i=1,2$, hence using the cyclic symmetry of the trace \begin{align*}
    \langle A\cdot U, A\cdot U'\rangle_{AJA^{-1}}&=\frac{1}{4}\Big(\tr(AU_1U_1'A^{-1})\tr\big((g_{AJA^{-1}}^{-1}(A^{-1})^*\sigma_1)^2\big) \\ & \ \ \ \ +\tr(AU_2U'_2A^{-1})\tr\big((g_{AJA^{-1}}^{-1}(A^{-1})^*\sigma_2)^2\big)\Big) \\ &=\frac{1}{4}\Big(\tr(U_1U_1')\tr\big((A(g_J^{-1}\sigma_1)A^{-1})^2\big) \\ & \ \ \ \ +\tr(U_2U'_2)\tr\big((A(g_J^{-1}\sigma_2)A^{-1})^2\big)\Big) \\ &=\frac{1}{4}\Big(\tr(U_1U_1')\tr\big((g_J^{-1}\sigma_1)^2\big) \\ & \ \ \ \ +\tr(U_2U'_2)\tr\big((g_J^{-1}\sigma_2)^2\big)\Big) \\ &=\langle U,U'\rangle_J \ .
\end{align*}
\end{proof}
\begin{lemma}
For any $(J,U)\in\maximal$ the following properties hold: \newline $\bullet$ $\langle UJ,U'J\rangle=\langle U,U'\rangle$; \newline $\bullet$ $\langle UJ,U'\rangle=-\langle U,U'J\rangle \ .$
\end{lemma}\begin{proof}
It is easy to show the above properties by using (\ref{eq:scalarproductU}) and $U_iJ=-JU_i$ for $i=1,2$. In fact,
\begin{align*}
    \langle UJ,U'J\rangle&=\frac{1}{2}\Big(\tr(U_1JU_1'J)+\tr(U_2JU_2'J)\Big) \\ &=-\frac{1}{2}\Big(\tr(U_1U_1'J^2)+\tr(U_2U_2'J^2)\Big) \\ &=\langle U,U'\rangle \ .
\end{align*}A similar computation shows the second claim.
\end{proof}
\subsection{The pseudo-K\"ahler metric}\label{sec:pseudokahler}
In this section we first describe tangent vectors to the vector bundle $\maximal\to\almost$ and then we explain how to adapt the scalar product we defined in (\ref{eq:scalarproductU}) on such elements, which will still be invariant for the $\SL(2,\R)$-action. Finally, we define a pseudo-Riemannian metric and a complex structure on $\maximal$ that will give rise to an $\SL(2,\R)$-invariant pseudo-K\"ahler metric on the subspace $\maximalflat$ we are interested in. 
\begin{proposition}\label{prop:decompositiondotU}
Let $(J,U)\in\maximal$ and let $\dot U:=g_J^{-1}\dot T$ be the unique $(1,3)$-tensor such that $g_J\big(\dot U(X,Y)Z,W\big)=\dot T(X,Y,Z,W)$ for all $X,Y,Z,W\in\R^2$. Then, an element $$(\dot J,\dot U)\in T_J\almost\times\Big(S_2(\R^2)\otimes\End(\R^2)\Big)$$ belongs to $T_{(J,U)}\maximal$ if and only if the following two conditions are satisfied: \newline $\bullet$ $\dot U_0=\dot{\widetilde U}_0+C(J,U,\dot J)$, where $\dot U_0$ is the endomorphism full trace-less part of $\dot U$, while the tensor $\dot{\widetilde U}_0$ is the trace-less part of $\dot U$ which is symmetric and $g_J$-traceless in the bi-linear form part and $C(J,U,\dot J)=2E\otimes e_2^*\otimes e_2^*$ in a local basis, with $E=U_1J\dot J\diag(-1,1)$; \newline $\bullet$ $\tr(\dot U(X,Y))=-\tr(J\dot JU(X,Y))$ for any $X,Y\in\R^2$.
\end{proposition}
\begin{proof}
The proof of the first requirement will become clear when, in Section \ref{sec:coordinatedescription}, we write the tangent vectors $(\dot J,\dot U)$ in coordinates. Regarding the second one, notice that $$\dot g_J=\rho(\cdot,\dot J\cdot)=-\rho(\cdot,J^2\dot J\cdot)=-g_J(\cdot,J\dot J\cdot) \ .$$ Moreover, for any $X,Y\in\R^2$ the endomorphism $U(X,Y)$ is trace-less, hence $$0=\delta\tr(U(X,Y))=\tr(\dot U(X,Y))-\tr(g^{-1}_J\dot{g_J}g^{-1}_JT) \ .$$Using the formula for the variation of the metric we obtain $g_J^{-1}\dot{g_J}=-J\dot J$ and so we conclude that $$\tr(\dot U(X,Y))=-\tr(J\dot JU(X,Y)) \ .$$
\end{proof}
\begin{remark}\label{rem:decompositionUdot}
For each $(J,U)\in\maximal$, the computations made in the previous proposition allow us to deduce the total variation of $U$ as follows: $$\delta U=\delta(g_J^{-1}T)=-g^{-1}_J\dot{g_J}g^{-1}_JT+g_J^{-1}\dot T=J\dot JU+\dot U \ .$$It is important to emphasize, however, that in the following we are going to consider the tensor $\dot U$ as the tangent vector to the fibre of $\maximal\to\almost$ and not $\delta U$. In addition, the  vector $\dot U$ can be further decomposed by looking at its endomorphism part as $$\dot U=\dot U_0+\dot U_{\text{tr}}, \quad \dot U_{\text{tr}}=-J\dot JU \ .$$
\end{remark}
Recall there is an $\SL(2,\R)$-action on $\maximal$ (see Section \ref{sec:linearalmost} and (\ref{eq:SLactiononU})), and if we look at the induced one on tangent vectors, we get a similar formula expressed as: \begin{equation}\label{eq:SLactiononDotU}
    A\cdot(\dot J,\dot U)=(A\dot JA^{-1},A\dot U(A^{-1}\cdot,A^{-1}\cdot)A^{-1}),\quad A\in\SL(2,\R) \ .
\end{equation}
Let $\{e_1,e_2\}$ be a $g_J$-orthonormal basis such that $Je_1=e_2$, and let $\{e_1^*,e_2^*\}$ be the dual basis. Let us denote with $\sigma_1$ and $\sigma_2$ the symmetric and $g_J$-traceless bi-linear forms $e_1^*\otimes e_1^*-e_2^*\otimes e_2^*$ and $e_1^*\otimes e_2^*+e_2^*\otimes e_1^*$, respectively. Then, $$\dot U_0=(\dot U_1)_0\otimes\sigma_1+(\dot U_2)_0\otimes\sigma_2+2E\otimes\big(e_2^*\otimes e_2^*\big), \quad E=U_1J\dot J\diag(-1,1) \ ,$$ where $(\dot U_1)_0$ and $(\dot U_2)_0$ are the trace-less part of the first order variations of the matrices $U_1=U(e_1,e_1)$ and $U_2=U(e_1,e_2)$, respectively. Furthermore, by pairing the endomorphism part with the bi-linear form part in a different way, we can rewrite \begin{equation}\label{eq:dotU}
    \dot U_0=(\dot U_1)_0\otimes\big(e_1^*\otimes e_1^*\big)+(\dot U_2)_0\otimes\sigma_2+\big(2E-(\dot U_1)_0\big)\otimes\big(e_2^*\otimes e_2^*\big) \ ,
\end{equation}so that it is easier to express the scalar product \begin{equation*}
    \langle\dot U_0,\dot U_0'\rangle=\frac{1}{4}\bigg(\tr((\dot U_1)_0(\dot U_1')_0)+2\tr((\dot U_2)_0(\dot U'_2)_0)+\tr(\big(2E-(\dot U_1)_0\big)\big(2E'-(\dot U_1')_0\big))\bigg)
\end{equation*}which is invariant by the $\SL(2,\R)$-action defined in (\ref{eq:SLactiononDotU}). The decomposition of the trace part is simpler $$\dot U_{\text{tr}}=(\dot U_1)_{\text{tr}}\otimes\sigma_1+(\dot U_2)_{\text{tr}}\otimes\sigma_2 $$ and the induced scalar product is given by $$\langle\dot U_{\text{tr}},\dot U'_{\text{tr}}\rangle=\frac{1}{2}\bigg(\tr((\dot U_1)_{\text{tr}}(\dot U_1')_{\text{tr}})+\tr((\dot U_2)_{\text{tr}}(\dot U_2')_{\text{tr}})\bigg) \ .$$At this point we have all the ingredients to introduce the pseudo-Riemannian metric, the complex structure, and symplectic form on $\maximal$. Let $f:[0,+\infty)\to(-\infty,0]$ be a smooth function such that: \newline $\bullet$ $f(0)=0$;
\newline $\bullet$ $f'(t)<0$ for any $t\ge 0$. \newline Then, we define the following symmetric bi-linear form on $T_{(J,U)}\maximal$ \begin{equation}\label{eq:pseudoriemannianmetric}
    \g_f\big((\dot J,\dot U),(\dot J',\dot U')\big):=\bigg(1-f\left(\frac{\|U\|^{2}}{2}\right)\bigg)\langle\dot J,\dot J'\rangle+\frac{1}{2}f'\left(\frac{\|U\|^{2}}{2}\right)\Big(\langle\dot U_0,\dot U_0'\rangle-\langle\dot U_{\text{tr}},\dot U_{\text{tr}}'\rangle\Big)
\end{equation}and the endomorphism \begin{equation}
    \i(\dot J,\dot U):=(-J\dot J,-\dot UJ-U\dot J) \ ,
\end{equation}where the products $\dot UJ$ and $U\dot J$ are performed at the level of the endomorphism part. By pairing the two objects together we obtain a bi-linear anti-symmetric form $$\ome_f(\cdot,\cdot):=\g_f(\cdot,\i\cdot)$$ which is explicitly given by \begin{equation*}
    \ome_f\big((\dot J,\dot U),(\dot J',\dot U')\big):=\bigg(f\left(\frac{\|U\|^{2}}{2}\right)-1\bigg)\langle\dot J,J\dot J'\rangle-\frac{1}{2}f'\left(\frac{\|U\|^{2}}{2}\right)\Big(\langle\dot U_0,\dot U_0'J\rangle-\langle\dot U_{\text{tr}},\dot U_{\text{tr}}'(\cdot,J\cdot)\rangle\Big)
\end{equation*}
\begin{remark}
From the formula of $\g_f$ and $\ome_f$ it is easily seen that they are invariant for the action of $\SL(2,\R)$ since they are written in terms of the scalar products defined on $(\dot J,\dot U)$. In particular, it follows that $\i$ is invariant for such an action as well.
\end{remark}
\begin{proposition}
The endomorphism $\i$ induces a complex structure on $\maximal$
\end{proposition}
\begin{proof}
We need to initially prove that $\i^2=-\Id$ and then the integrability of $\i$. The first claim is just a computation \begin{align*}
    \i^2(\dot J,\dot U)&=\i(-J\dot J,-\dot UJ-U\dot J) \\ &=(J^2\dot J, -(-\dot UJ-U\dot J)J-U(-J\dot J)) \\ &=(-\dot J,-\dot U+U\dot JJ+UJ\dot J) \tag{$\dot J\in T_J\almost$}\\ &=(-\dot J,-\dot U) \ .
\end{align*}Since the almost-complex structure $\i$ acts on tangent vectors $\dot J$ as that of Lemma \ref{lem:kahleronalmost}, regarding its integrability we only need to show that under the isomorphism $\Phi:\maximal\to\quartic$ explained in Proposition \ref{prop:isomorphismmaximalandquartic} it corresponds to the multiplication by $-i$ on the fibre of $\quartic$. Since the latter is integrable the former is integrable as well. This consists of computing the real part of the variation $-i\dot q$, where $q$ is a quartic $J$-holomorphic differential on the torus. First notice that, by Theorem \ref{thm:tensorsandquarticdifferentials}, we have $$\Ree(-iq)=\Im(q)=-T(\cdot,\cdot,\cdot,J\cdot)=-T(\cdot,\cdot,J\cdot,\cdot) \ .$$ Then, its variation satisfies \begin{align*}
    \widetilde T(\cdot,\cdot,\cdot,\cdot):=\Ree(-i\dot q)=-\dot T(\cdot,\cdot,J\cdot,\cdot)-T(\cdot,\cdot,\dot J\cdot,\cdot) \ .
\end{align*}Let us denote with $\widetilde U$ the associated $(1,3)$-tensor, namely $\widetilde U=g_J^{-1}\widetilde T$, so that \begin{align*}
    g_J\big(\widetilde U(X,Y)Z,W\big)&=\widetilde T(X,Y,Z,W) \\ &=-\dot T(X,Y,JZ,W)-T(X,Y,\dot JZ,W) \\ &=-g_J\big(\dot U(X,Y)JZ+U(X,Y)\dot JZ,W\big), \quad \text{for any} \ X,Y,Z,W\in\R^2 \ .
\end{align*}In other words $\widetilde U=-\dot UJ-U\dot J$, and we obtain the claim.
\end{proof}
\begin{lemma}
    For any $(\dot J,\dot U)\in T_{(J,U)}\maximal$, we have the following decomposition \begin{equation}\label{eq:decompositioncomplexstructure}
        \dot UJ+U\dot J=\dot U_0J+\dot U_{\text{tr}}(\cdot,J\cdot) \ ,
    \end{equation}where $\dot U_0J$ and $\dot U_{\text{tr}}(\cdot,J\cdot)$ represents the endomorphism full trace-less part and trace part, respectively.
\end{lemma}
\begin{proof}
First recall that by Proposition \ref{prop:decompositiondotU} we can write $\dot U=\dot U_0+\dot U_{\text{tr}}$, hence $$\dot UJ+U\dot J=\dot U_0J+\dot U_{\text{tr}}J+U\dot J \ .$$ We obtain the claim if we show that $\dot U_{\text{tr}}J+U\dot J=\dot U_{\text{tr}}(\cdot,J\cdot)$. Let $\{e_1,e_2\}$ be a $g_J$-orthonormal basis and let $\{e_1^*,e_2^*\}$ the dual one. Let $\sigma_1,\sigma_2$ be the basis in (\ref{eq:basissymmetricbilinear}), then $$\dot U_{\text{tr}}J=(\dot U_1)_{\text{tr}}J\otimes\sigma_1+(\dot U_2)_{\text{tr}}J\otimes\sigma_2, \quad (\dot U_i)_{\text{tr}}=\frac{1}{2}\tr(\dot U_i)\mathds{1} \ \text{for} \ i=1,2 \ .$$ Moreover, using Lemma \ref{lem:dotJ}, we get \begin{align*}
    U\dot J&=U_1\dot J\otimes\sigma_1+U_2\dot J\otimes\sigma_2 \\ &=\frac{1}{2}\bigg(\Big(\tr(U_1\dot J)\mathds{1}-\tr(JU_1\dot J)J\Big)\otimes\sigma_1+\Big(\tr(U_2\dot J)\mathds{1}-\tr(JU_2\dot J)J\Big)\otimes\sigma_2\bigg) \ .
\end{align*}Next, $JU_2=U_1$ implies that $\tr(U_1\dot J)=\tr(JU_2\dot J)=\tr(\dot U_2)$ (see Proposition \ref{prop:decompositiondotU}) and $\tr(U_2\dot J)=-\tr(JU_1\dot J)=-\tr(\dot U_1)$. Summing up the two terms we are interested in, we can conclude that $$\dot U_{\text{tr}}J+U\dot J=\frac{1}{2}\bigg(\tr(\dot U_2)\mathds{1}\otimes\sigma_1-\tr(\dot U_1)\mathds{1}\otimes\sigma_2\bigg) \ ,$$ and it is not difficult to see that the last term is equal to $\dot U_{\text{tr}}(\cdot,J\cdot)$.
\end{proof}
\begin{proposition}\label{prop:compatible}
The symmetric bi-linear form $\g_f$ is compatible with the complex structure $\i$
\end{proposition}
\begin{proof}
It boils down in proving that for each $(J,U)\in\maximal$ the following holds: $$(\g_f)_{(J,U)}\big(\i(\dot J,\dot U),\i(\dot J',\dot U')\big)=(\g_f)_{(J,U)}\big((\dot J,\dot U),(\dot J',\dot U')\big) \ .$$
By definition of the complex structure we obtain \begin{align*}
    (\g_f)_{(J,U)}\big(\i(\dot J,\dot U),\i(\dot J',\dot U')\big)&=(1-f)\langle J\dot J,J\dot J'\rangle +\frac{f'}{2}\langle(\dot UJ+U\dot J)_0,(\dot U'J+U\dot J')_0 \rangle
\\  & \ \ \ \ \  -\frac{f'}{2}\langle(\dot UJ+U\dot J)_{\text{tr}},(\dot U'J+U\dot J')_{\text{tr}} \rangle \ , \end{align*} where the functions $f$ and $f'$ are evaluated at $\frac{\| U\|^2}{2}$, hence they are not changed after applying the complex structure. The first term can be simplified using the relation $J\dot J=-\dot JJ$, indeed $$\langle J\dot J,J\dot J'\rangle=\tr(J\dot J J\dot J')=-\tr(J^2\dot J\dot J')=\langle \dot J,\dot J'\rangle \ .$$ Regarding the second term, we use the decomposition (\ref{eq:decompositioncomplexstructure}), so that \begin{align*}
 \langle(\dot UJ+U\dot J)_0,(\dot U'J+U\dot J')_0 \rangle&=\langle\dot U_0J,\dot U_0'J\rangle \\ &=\frac{1}{2}\bigg(\tr((\dot U_1)_0J(\dot U_1')_0J)+\tr((\dot U_2)_0J(\dot U_2')_0J)\bigg)+\tr(EJE'J) \\ &=\frac{1}{2}\bigg(\tr((\dot U_1)_0(\dot U_1')_0)+\tr((\dot U_2)_0(\dot U_2')_0)\bigg)+\tr(EJE'J) \ ,
\end{align*}where we used that the endomorphism part of $(\dot U_i)_0$ and $(\dot U_i')_0$ is $g_J$-symmetric and trace-less. Moreover, since $E=U_1J\dot J\diag(-1,1)$ and $E'=U_1J\dot J'\diag(-1,1)$ we conclude that $\tr(EJE'J)=\tr(EE')$ which is exactly what we aimed for. Finally, for the last term, we use once again (\ref{eq:decompositioncomplexstructure}) and the expression in a $g_J$-orthonormal basis $$\dot U_{\text{tr}}(\cdot,J\cdot)=(\dot U_2)_{\text{tr}}\otimes(e_2^*\otimes e_2^*-e_1\otimes e_1^*)+(\dot U_1)_{\text{tr}}\otimes(e_1^*\otimes e_2^*+e_2^*\otimes e_1^*) \ .$$ In fact, \begin{align*}
    \langle(\dot UJ+U\dot J)_{\text{tr}},(\dot U'J+U\dot J')_{\text{tr}} \rangle&=\langle\dot U_{\text{tr}}(\cdot,J\cdot),\dot U'_{\text{tr}}(\cdot,J\cdot)\rangle \\ &=\frac{1}{2}\bigg(\tr((\dot U_2)_{\text{tr}}(\dot U_2')_{\text{tr}})+\tr((\dot U_1)_{\text{tr}}(\dot U_1')_{\text{tr}})\bigg) \\ &=\langle\dot U_{\text{tr}},\dot U_{\text{tr}}'\rangle \ .
\end{align*}
\end{proof}
\begin{theorem}\label{thm:pseudoKahler}
The triple $(\g_f,\i,\ome_f)$ defines an $\SL(2,\R)$-invariant pseudo-K\"ahler metric on $\maximal$ which restricts to a $\mathrm{MCG}(T^2)$-invariant pseudo-K\"ahler metric on $\maximalflat$
\end{theorem}
\begin{proof}
The only remaining items to be shown to obtain an $\SL(2,\R)$-invariant pseudo-K\"ahler metric on $\maximal$ are the non-degeneracy of $\g_f$ and the closedness of $\ome_f$. Their proof is deferred to Proposition \ref{prop:Gnondegenereomegachiusa} as it is necessary a computation in coordinates. \newline The complex structure $\i$ preserves the tangent to the $0$-section in $T_{(J,U)}\maximal$ since $\i_{(J,0)}(\dot J,0)=(-J\dot J,0)$. We deduce that $\i$ still defines a complex structure on the complement of the $0$-section in $\maximal$, which is identified with $\maximalflat$ according to Theorem \ref{thm:flatmaximalandquartic} and Proposition \ref{prop:isomorphismmaximalandquartic}. The same argument holds for $\g_f$ and $\ome_f$ so that we get a well-defined pseudo-K\"ahler metric on $\maximalflat$ which is $\mathrm{MCG}(T^2)\cong\SL(2,\Z)$-invariant.
\end{proof}
\begin{remark}
It should be noted that our construction holds for any choice of a smooth function $f:[0,+\infty)\to(-\infty,0]$ such that $f(0)=0$ and $f'(t)<0$ for any $t\ge 0$. So, in fact, we have proved the existence of an infinite family of pseudo-K\"ahler metrics on $\maximal$, and thus on $\maximalflat$.
\end{remark}
\subsection{Coordinate description}\label{sec:coordinatedescription}
In order to finish the proof of Theorem \ref{thm:pseudoKahler}
it only remains to show that the symmetric tensor $\g_f$ and the $2$-form $\ome_f$ on $\maximal$ are non-degenerate and closed, respectively. In order to do so, we are going to write their expression in local coordinates. First of all, in light of Lemma \ref{lem:Trautweinkahlerhyperbolic} we can identify $\mathcal T(T^2)$ with $\Hyp$ so that the diffeomorphism is $\SL(2,\R)$-equivariant. Hence, the total space of the holomorphic vector bundle of quartic differentials $\quartic\to\mathcal T(T^2)$ is diffeomorphic to $\Hyp\times\C$, where $\C$ is the copy of the fibre over a point $z\in\Hyp$. We can define an $\SL(2,\R)$-action on $\Hyp\times \C$ by
\begin{equation*}\label{Sl2Raction}\begin{pmatrix}
a & b \\ c & d
\end{pmatrix}\cdot(z,w):=\bigg(\frac{az+b}{cz+d}, (cz+d)^4w\bigg),\qquad \text{with} \ (z,w)\in\Hyp\times\C,\quad ad-bc=1 \ . \end{equation*}
Moreover, there is a metric on the fiber induced by the norm 
$$|w|_z^2=\Imm(z)^4|w|^2 \qquad \text{for} \ z\in\Hyp, w\in\C \ .$$ 
Given $J\in\almost$, let us define the space of $J$-complex symmetric quadri-linear forms by \begin{align*}S_4(\R^2, J):&=\{\gamma:\big(\R^2\big)^{\otimes^4}\longrightarrow\C \ | \ \gamma \ \text{ is symmetric and} \ (J,i)-\text{quadri-linear}\} \\ &\cong\{\tau:\R^2\to\C \ | \ \text{for all} \ \alpha,\beta\in\R \ \text{and} \ v\in\R^2 \ \text{it holds} \ \tau(\alpha v+\beta Jv)=(\alpha+i\beta)^4\tau(v)\} \ . \end{align*}This space can be seen as the fiber of a complex line bundle $\mathcal{L}(\R^2)\to\almost$ endowed with a natural $\SL(2,\R)$-action given by $$A\cdot (J,\gamma):=(AJA^{-1}, (A^{-1})^*\gamma),\qquad \ \text{for} \ A\in\SL(2,\R) \ .$$It is not difficult to see that the line bundle $\mathcal{L}(\R^2)$ can be identified with $\maximal$. In particular, each fiber $S_4(\R^2, J)$ is endowed with a scalar product from the one on $\maximal_J$ defined in (\ref{eq:scalarproductU}).\begin{lemma}[{\cite[Lemma 5.2.1]{trautwein2018infinite}}]\label{lem:Trautweinfibremap}Let us consider the map $\varphi:\quartic\to\Hom\Big(\big(\R^2\big)^{\otimes^4}, \C\Big)$ given by\begin{align*}   \varphi(z,w) : \ &\R^2\longrightarrow\C \\ & v\mapsto \bar w(v_1-\bar zv_2)^4\end{align*}and let $j:\Hyp\to\almost$ be the K\"ahler isometry defined in (\ref{eq:kahlerisometryalmost}). Then, the following holds: \begin{itemize}   \item[$\bullet$]$\varphi(z,w)\in S_4(\R^2,j(z)), \ \text{for all} \  (z,w)\in\quartic$.  \item[$\bullet$]The fibre map $\varphi(z,\cdot) : \quartic_z\cong\C\to S_4(\R^2,j(z))$ is a complex anti-linear isometry for every $z\in\Hyp$. \item[$\bullet$] The bundle map $(j,\varphi): \quartic\to\mathcal{L}(\R^2)$ is a $\SL(2,\R)$-equivariant bijection.\end{itemize}\end{lemma}
At this point it easy to compute in coordinates the tensors $T$ and $U=g_J^{-1}T$ and their respective variations: $\dot T$ and $\dot U=g_J^{-1}\dot T$, by using the last lemma and the isomorphism $\maximal\cong\quartic$. Let $z=x+iy$ and $w=u+iv$ be the complex coordinates on $\Hyp$ and $\C$ respectively, then the bundle map $(j,\varphi)$ is given by \begin{align*}&\Hyp\times\C\longrightarrow\maximal \\ &(z,w)\longmapsto\big(j(z),T_{(z,w)}\big)\end{align*}where $T_{(z,w)}=\Ree(q_{(z,w)})$ with $q_{(z,w)}=\widebar w(\dx_0-\widebar z\dy_0)^4$ (Theorem \ref{thm:tensorsandquarticdifferentials}) and where $(x_0,y_0)$ are the coordinates on $\R^2$. Since $\SL(2,\R)$ acts transitively on $\Hyp$, it is enough to compute the tensors at the point $(i,w)\equiv(0,1,u,v)$ for a generic $w\in \C$. What we get is the following\begin{align*}
    &T_{1111}(z,w)=u,\qquad T_{1112}(z,w)=-xu+yv,\qquad T_{1122}(z,w)=ux(^2-y^2)-2xyv, \\ &T_{1222}(z,w)=-ux^3-vy^3+3(uy^2x+x^2yv), \\ & T_{2222}(z,w)=u(x^4-6x^2y^2+y^4)+4v(xy^3-x^3y) \ .\end{align*}
The remaining components are determined by the five above since $T$ is totally-symmetric. Its variation $\dot T_{(i,w)}$ at $(i,w)$ is
\begin{align*}
    &\dot T_{1111}(i,w)=\dot u,\qquad \dot T_{1112}(i,w)=-u\dot x+\dot v+v\dot y,\qquad \dot T_{1122}(i,w)=-\dot u-2(u\dot y+v\dot x), \\ &\dot T_{1222}(i,w)=-\dot v+3(u\dot x-v\dot y), \qquad\dot T_{2222}(i,w)=\dot u+4(u\dot y+v\dot x) \ . \end{align*}
The corresponding tensor $U=g_J^{-1}T$ computed at $(i,w)$ is then \begin{align*}
    U_{(i,w)}=\begin{pmatrix}
    u & v \\ v & -u
    \end{pmatrix}\otimes\big(\mathrm dx_0\otimes\mathrm{d}x_0-\mathrm dy_0\otimes\mathrm dy_0\big)+\begin{pmatrix}
    v & -u \\ -u & -v
    \end{pmatrix}\otimes\big(\dx_0\otimes\dy_0+\dy_0\otimes\dx_0\big) \ .
\end{align*}
Its variation $\dot U$ will be given in terms of its endomorphism trace-free and trace part at the point $(i,w)$
\begin{align*}\label{dotpickformcoordinates}
     (\dot U_0)_{(i,w)}&=\begin{pmatrix}
    \dot u+u\dot y+v\dot x & -u\dot x+\dot v+v\dot y \\ -u\dot x+\dot v+v\dot y & -\dot u-u\dot y-v\dot x
    \end{pmatrix}\otimes\big(\dx_{0}\otimes\dx_0\big) \\ & \ \ \ \ \ +\begin{pmatrix}
    \dot v+2(v\dot y-u\dot x) & -\dot u-2(u\dot y+v\dot x) \\ -\dot u-2(u\dot y+v\dot x) & -\dot v+2(u\dot x-v\dot y)
    \end{pmatrix}\big(\dx_{0}\otimes\dy_0+\dy_0\otimes\dx_0\big) \\ & \ \ \ \ \ + \begin{pmatrix}
        -\dot u-3u\dot y-3v\dot x & -\dot v-3v\dot y+3u\dot x \\ -\dot v -3v\dot y +3u\dot x & +\dot u +3u\dot y +3v\dot x
    \end{pmatrix}\otimes\big(\dy_0\otimes\dy_0\big) \ ,
\end{align*}where, according to (\ref{eq:dotU}), we have $$\begin{pmatrix}
        -\dot u-3u\dot y-3v\dot x & -\dot v-3v\dot y+3u\dot x \\ -\dot v -3v\dot y +3u\dot x & +\dot u +3u\dot y +3v\dot x
    \end{pmatrix}=2E-(\dot U_1)_0 \ .$$ The endomorphism trace part is\begin{align*}
    (\dot U_{\tr})_{(i,w)}&=\begin{pmatrix}
    -u\dot y-v\dot x & 0 \\ 0 & -u\dot y-v\dot x
    \end{pmatrix}\otimes\big(\mathrm dx_0\otimes\mathrm{d}x_0-\mathrm dy_0\otimes\mathrm dy_0\big) \\ & \ \ \ \ +\begin{pmatrix}
    u\dot x-v\dot y & 0 \\ 0 & u\dot x-v\dot y
    \end{pmatrix}\otimes\big(\dx_0\otimes\dy_0+\dy_0\otimes\dx_0\big) \ . 
\end{align*}
Thanks to this expression in coordinates and together with the action of $\SL(2,\R)$ on $\Hyp\times\C$, we are now able to write the pseudo-metric $\g_f$ and the symplectic form $\ome_f$ at the point $(z,w)$ (see \cite[\S 3.1]{rungi2021pseudo} for a similar computation). Let $\{\frac{\partial}{\partial x},\frac \partial{\partial y},\frac\partial{\partial u},\frac\partial{\partial v}\}$ be a real basis of the tangent space of $\Hyp\times\C$ with its dual basis $\{\dx,\dy,\du,\devu\}$, then the expression (\ref{eq:pseudoriemannianmetric}) becomes 
\begin{equation*}
    (\g_f)_{(z,w)}=\begin{pmatrix}
    \frac 1{y^2}\big(1-f+4(u^2+v^2)y^4f'\big) & 0 & 2f'vy^3 & -2f'uy^3 \\ 0 & \frac 1{y^2}\big(1-f+4(u^2+v^2)y^4f'\big) & 2f'uy^3 & 2f'vy^3 \\ 2f'vy^3 & 2f'uy^3 &  f'y^4 & 0 \\ -2f'uy^3 & 2f'vy^3 & 0 & f'y^4
\end{pmatrix}\end{equation*}\begin{align*}(\ome_f)_{(z,w)}=&\bigg(-1+f-4f'y^4(u^2+v^2)\bigg)\frac{\dx\wedge\dy}{y^2}-f'y^4\du\wedge\devu \\&-2y^3f'\bigg(u(\dx\wedge\du+\dy\wedge\devu)+v(\du\wedge\dy-\devu\wedge\dx)\bigg)\end{align*}
where the functions $f,f'$ are evaluated in: $$\frac 1{2}\Big\vl\Big\vl U_{(z,w)}\Big\vl\Big\vl^2_{j(z)}=\big\vl\big\vl q_{(z,w)}\big\vl\big\vl^2_{j(z)}=y^4(u^2+v^2) \ .$$
The matrix associated with the complex structure $\i_{(z,w)}: T_{(z,w)}\big(\Hyp\times\C\big)\to T_{(z,w)}\big(\Hyp\times\C\big)$ in the basis $\{\frac{\partial}{\partial x},\frac \partial{\partial y},\frac\partial{\partial u},\frac\partial{\partial v}\}$ is \begin{equation*}
    \i_{(i,w)}=\begin{pmatrix}
    J_0 & 0_{2\times 2} \\ 0_{2\times 2} & J_0
    \end{pmatrix}=\begin{pmatrix}
    0 & -1 & 0 & 0 \\ 1 & 0 & 0 & 0 \\ 0 & 0 & 0 & -1 \\ 0 & 0 & 1 & 0
    \end{pmatrix} \ . 
\end{equation*}
\begin{proposition}\label{prop:Gnondegenereomegachiusa}
For any $(z,w)\in\Hyp\times\C$ we have $$\det\big(\g_f\big)_{(z,w)}\neq 0, \qquad \mathrm d\big(\ome_f\big)_{(z,w)}=0 \ .$$ In particular, $\g_f$ is non-degenerate and $\ome_f$ is closed on $\Hyp\times\C\cong\maximal$.
\end{proposition}\begin{proof}
The tensor $\g_f$ can be written as: $$(\g_f)_{(z,w)}=\begin{pmatrix}
\Theta & \Xi \\ \Gamma & \Delta
\end{pmatrix}$$where $\Theta,\Xi,\Gamma,\Delta$ are $2\times 2$ matrices with $$\Theta=\frac 1{y^2}\big(1-f+4y^4(u^2+v^2)f'\big)\dsone_{2\times 2}, \qquad \Delta=y^4f'\dsone_{2\times 2} \ .$$
Hence, $\Xi$ and $\Gamma$ both commute with $\Theta$ and $\Delta$. In this case, we get an easy formula for the determinant, namely $\det\Big((\g_f)_{(z,w)}\Big)=\det(\Theta\Delta-\Xi\Gamma)$, where $$\Theta\Delta=y^2f'\Big(1-f'+4y^4f'(u^2+v^2)\Big)\dsone_{2\times 2}, \qquad \Xi\Gamma=4y^6(f')^2(u^2+v^2)\dsone_{2\times 2}$$which gives $$\det\Big((\g_f)_{(z,w)}\Big)=y^4(f')^2(1-f)^2 \ . $$
The right hand side of the last equation is always non-zero thanks to the property of the function $f$, hence $(\g_f)_{(z,w)}$ is non-degenerate at each point $(z,w)\in\Hyp\times\C$.\newline It only remains to prove that $(\mathrm{d}\ome_f)_{(z,w)}=0$ for each $(z,w)\in\Hyp\times\C$. By using directly the expression in coordinate, we get: \begin{itemize}
    \item Coefficient $\dy\wedge\du\wedge\devu$:\begin{align*}
         &-4y^3f'\dy\wedge\du\wedge\devu-4y^7f''(u^2+v^2)\dy\wedge\du\wedge\devu-4y^7f''u^2\du\wedge\dy\wedge\devu \\& -2y^3f'\du\wedge\dy\wedge\devu-4y^7f''v^2\devu\wedge\du\wedge\dy-2y^3f'\devu\wedge\du\wedge\dy=0
    \end{align*} \item Coefficient $\dx\wedge\du\wedge\devu$:\begin{equation*}
-4y^7f''uv\devu\wedge\dx\wedge\du+4y^7f''uv\du\wedge\devu\wedge\dx=0
    \end{equation*}\item Coefficient $\dx\wedge\dy\wedge\devu$:\begin{align*}
        &2y^2f'v\devu\wedge\dx\wedge\dy-6y^6f''(u^2+v^2)v\devu\wedge\dx\wedge\dy-6y^2f'v\devu\wedge\dx\wedge\dy \\ &+6y^2f'v\dy\wedge\devu\wedge\dx+6y^6f''v(u^2+v^2)\dy\wedge\devu\wedge\dx=0
    \end{align*}
    \item Coefficient $\dx\wedge\dy\wedge\du$: \begin{align*}
        &2y^2f'u\du\wedge\dx\wedge\dy-6y^6f'u(u^2+v^2)\du\wedge\dx\wedge\dy-6y^2f'u\du\wedge\dx\wedge\dy \\ &-6y^2f'u\dy\wedge\dx\wedge\du-6y^6f''(u^2+v^2)u\dy\wedge\dx\wedge\du=0
    \end{align*}
\end{itemize}
\end{proof}
\begin{remark}
At this point, once it is shown that $\g_f$ is non-degenerate, it is easily seen that it has neutral signature $(2,2)$. In fact, it restricts to a positive definite metric on the copy of the hyperbolic plane $\Hyp\times\{0\}\subset\Hyp\times\C$ and to a negative definite metric on the copy of the fibre $\{i\}\times\C\subset\Hyp\times\C$.
\end{remark}

\section{Hamiltonian actions and moment maps}
\noindent In this section we introduce two Hamiltonian actions on $\mathcal{MF}(T^{2})$ that will endow the space of complete equivariant flat maximal surfaces in $\mathbb{H}^{2,2}$ with the structure of a completely integrable system with complete fibers. This will allow us to find global Darboux coordinates for $\ome_{f}$ that have interesting geometric meanings. \\

\noindent Let us first recall some basic definitions regarding symplectic actions and moment maps.
\begin{definition}
Let $(M,\omega)$ be a symplectic manifold and let $G$ be a Lie group acting on $M$. Let $\psi_g:M\to M$ be the map $\psi_g(p):=g\cdot p$, then we say the group $G$ acts by symplectomorphisms on $(M,\omega)$ if $\psi_g^*\omega=\omega$ for all $g\in G$.
\end{definition}  
\begin{definition}\label{def:momentmap}
Let $G$ be a Lie group, with Lie algebra $\mathfrak g$, acting on a symplectic manifold $(M,\omega)$ by symplectomorphisms. We say the action is \emph{Hamiltonian} if there exists a smooth function $\mu:M\to\mathfrak g^*$ satisfying the following properties: \begin{itemize}
    \item[(i)] The function $\mu$ is equivariant with respect to the $G$-action on $M$ and the co-adjoint action on $\mathfrak g^*$, namely \begin{equation}
        \mu_{g\cdot p}=\Ad^*(g)(\mu_p):=\mu_p\circ \Ad(g^{-1})\in\mathfrak g^* \ .
    \end{equation}
    \item[(ii)]Given $\xi\in\mathfrak g$, let $X_\xi$ be the vector field on $M$ generating the action of the $1$-parameter subgroup generated by $\xi$, i.e. $X_\xi=\frac{\mathrm d}{\mathrm dt}\text{exp}(t\xi)\cdot p |_{t=0}$. Then, for every $\xi\in\mathfrak g$ we have\begin{equation}
        \mathrm d\mu^{\xi}=\iota_{X_\xi}\omega=\omega(X_\xi,\cdot)
    \end{equation}where $\mu^\xi:M\to\R$ is the function $\mu^\xi(p):=\mu_p(\xi)$.
\end{itemize}A map $\mu$ satisfying the two properties above is called a \emph{moment map} for the Hamiltonian action.
\end{definition}

\subsection{The circle action}\label{sec:circleaction}
The space of holomorphic quartic differentials on $T^{2}$ has a natural circle action given by 
\begin{align*}
    S^{1}\times \mathcal{Q}^{4}(T^2) &\rightarrow \mathcal{Q}^{4}(T^2) \\
    (\theta, J, q) &\mapsto (J, e^{-i\theta}q) \ .
\end{align*}
By taking the real part of $q$, we have an induced circle action on $\mathcal{M}(\R^{2})$
\begin{align*}
    \Psi_{\theta}: \mathcal{M}(\R^2) &\rightarrow \mathcal{M}(\R^2) \\
    (J, U) &\mapsto (J, \cos(\theta)U-\sin(\theta)UJ) \ .
\end{align*}

\begin{theorem} The function 
\[
    H(J,U)=\frac{1}{2}f\left(\frac{\|U\|^{2}}{2}\right) \ 
\]
is a Hamiltonian for the circle action on $(\mathcal{M}(\R^{2}), \ome_{f})$. Moreover, $\Phi_{\theta}^{*}\g_{f}=\g_{f}$.
\end{theorem}
\begin{proof}
The infinitesimal generator of the circle action is
\[
    X_{(J,U)} = \frac{d}{d\theta}_{\big|_{\theta=0}}\Psi_{\theta}(J,U)=(0,-UJ) \ . 
\]  
Now, at a point $(J,U) \in \mathcal{MF}(T^{2})$, we have
\begin{align*}
    \mathcal{L}_{X}\ome_{f}((\dot{J}, \dot{U})) &= \ome_{f}( (\dot{J}, \dot{U}) , (0,UJ)) 
    = \g_{f}((\dot{J}, \dot{U}) , \i(0,UJ)) \\
    &= \g_{f}((\dot{J}, \dot{U}) , (0,-UJ^{2})) 
    = \g_{f}((\dot{J}, \dot{U}) , (0,U)) \\
    &= \frac{f'}{2} \langle \dot{U}_{0}, U \rangle
\end{align*}
where the function $f'$ is evaluated at $\|U\|^{2}/2$. It remains to check that this last expression coincides with the differential of $H$. To this aim, we first compute the derivative of the norm of $U$ in the direction of $(\dot{J}, \dot{U})$:
\begin{align*}
    (\|U\|_{J}^{2})'&= 2\langle U, U' \rangle \\ 
    &= 2\langle U, \dot{U} \rangle + 2\langle U, J\dot{J}U \rangle \tag{$\langle U,\dot{U}_{tr}\rangle=0$} \\
    &= 2\langle U, \dot{U}_{0}\rangle +2\langle U, J\dot{J}U\rangle \\
    & = 2\langle U, \dot{U}_{0}\rangle 
\end{align*}
because 
\[
    2\langle U, J\dot{J}U\rangle = \trace(U_{1}J\dot{J}U_{1})+\trace(U_{2}J\dot{J}U_{2})
\]
and both traces vanish by Lemma \ref{lem:dotJ}. As a consequence, 
\[
    \mathrm dH_{(J,U)}(\dot{J},\dot{U}) = \frac{1}{2}f'\left(\frac{\|U\|^{2}}{2}\right) \cdot \frac{1}{2} (\|U\|_{J}^{2})' = \frac{1}{2}f'\left(\frac{\|U\|^{2}}{2}\right) \langle U, \dot{U}_{0} \rangle
\]
so the function $H$ is an Hamiltonian for the circle action.\\
Let us now prove that $\Psi_{\theta}$ preserves the metric. First we note that the differential of $\Psi_{\theta}$ at a point $(J,U)\in \mathcal{M}(\R^{2})$ is
\begin{align*}
    \mathrm d\Psi_{\theta}(\dot{J}, \dot{U}) &= (\dot{J}, \cos(\theta)\dot{U}-\sin(\theta)(\dot{U}J+U\dot{J})) \\
        &=(\dot{J}, \cos(\theta)\dot{U}+\sin(\theta)\i(\dot{U}))
\end{align*}
and then we compute
\begin{align*}
    (\Psi_{\theta}^{*})\g_{f}((\dot{J}, \dot{U}), &(\dot{J}',\dot{U}')) = \g_{f}(\mathrm d\Psi_{\theta}(\dot{J}, \dot{U}), \mathrm d\Psi_{\theta}(\dot{J}',\dot{U}')) \\
    &=\g_{f}((\dot{J}, \cos(\theta)\dot{U}+\sin(\theta)\i(\dot{U})), (\dot{J}', \cos(\theta)\dot{U}'+\sin(\theta)\i(\dot{U}'))) \\
    &= (1-f)\langle \dot{J}, \dot{J}' \rangle +\frac{f'}{2}\Bigg[\cos^{2}(\theta)\bigg(\langle \dot{U}_{0}, \dot{U}_{0}' \rangle - \langle \dot{U}_{tr}, \dot{U}_{tr}' \rangle\bigg) \\
    & \ \  +\sin^{2}(\theta)\bigg(\langle \i(\dot{U})_{0}, \i(\dot{U}')_{0} \rangle - \langle \i(\dot{U})_{tr}, \i(\dot{U}')_{tr} \rangle \bigg) \\
    & \ \  +\cos(\theta)\sin(\theta)\bigg( \langle \i(\dot{U})_{0}, \dot{U}_{0}' \rangle + \langle \dot{U}_{0}, \i(\dot{U}')_{0} \rangle - \langle \i(\dot{U})_{tr}, \dot{U}_{tr}' \rangle - \langle \dot{U}_{tr}, \i(\dot{U}')_{tr}  \rangle  \bigg) \Bigg] \ .
\end{align*}
We already observed that $\i(\dot{U})_{0}=-\dot{U}_{0}J$ and $\langle \dot{U}_{0}J, \dot{U}_{0}'J\rangle = \langle \dot{U}_{0}, \dot{U}_{0}' \rangle$. Similarly, $\i(\dot{U})_{tr}=-\dot{U}_{tr}(\cdot, J\cdot)$ and $\langle \dot{U}_{tr}(\cdot, J\cdot), \dot{U}_{tr}'(\cdot, J\cdot) \rangle = \langle \dot{U}_{tr}, \dot{U}_{tr}' \rangle$. Therefore, in order to complete the proof, it is sufficient to check that the term multiplied by $\cos(\theta)\sin(\theta)$ vanishes. This follows from a direct computation similar to Proposition \ref{prop:compatible}. 
\end{proof}
\subsection{The \texorpdfstring{$\SL(2,\R)$}{SL(2,R)}-action}\label{sec:SLaction}
In Section \ref{sec:vector_bundle}, we introduced an $\SL(2,\R)$-action on $\mathcal{M}(\R^{2})$ (see expression \ref{eq:SLactiononU}). We now prove that it is Hamiltonian for the symplectic forms $\ome_{f}$.

\begin{proposition}\label{prop:momentmapSL2R} The $\mathrm{SL}(2,\R)$-action on $\mathcal{M}(\R^{2})$ is Hamiltonian with moment map
\[
    \mu_{X}(J,U)=(1-f)\trace(JX)
\]
where $f$ is evaluated at $\|U\|^{2}/2$ and $X\in \mathfrak{sl}(2,\R)$.    
\end{proposition}
\begin{proof} We start by computing the infinitesimal generator of this action for a given $X\in\Lsl(2,\R)$ 
\[
V_X(J,U)=\frac{\mathrm d}{\mathrm dt}(e^{tX}Je^{-tX}, (e^{-tX})^{*}T)|_{t=0}
\]
The first component is clearly equal to $XJ-JX=[X,J]$. For the second component, define $P_t:=e^{tX}$. We need to compute  
\begin{align*}
    \frac{\mathrm d}{\mathrm dt}T((P_t)^{-1}\cdot,(P_t)^{-1}\cdot,&(P_t)^{-1}\cdot, (P_t)^{-1}\cdot)|_{t=0}= \\
    &-T(X\cdot,\cdot,\cdot, \cdot)-T(\cdot,X\cdot,\cdot,\cdot)-T(\cdot,\cdot,X\cdot,\cdot)-T(\cdot, \cdot, \cdot, X\cdot) \ . 
     \end{align*}
     If $\widetilde T(\cdot,\cdot,\cdot,\cdot)$ is defined as the right hand side term of the equation above, then the associated $(3,1)$-tensor $\widetilde U=g_{J}^{-1}\widetilde{T}$ satisfies \begin{align*}
         g_J(\widetilde U(Y,Z)W,R)&=\widetilde T(Y,Z,W,R) \\ &=-T(XY,Z,W,R)-T(Y,XZ,W,R)-T(Y,Z,XW,R)-T(Y,Z,W,XR) \\ &=-g_J(U(XY,Z)W,R)-g_J(U(Y, XZ)W, R)-g_J(U(Y,Z)XW,R)+ \\
         & \ \ \ -g_{J}(U(Y,Z)W, XR)\\ &=-g_J(U(XY,Z)W+U(Y,XZ)W+U(Y,Z)XW+X^*U(Y,Z)W, R)
     \end{align*}
     for all $Y,Z,W,R\in\Gamma(T\R^2)$, where $X^*$ denotes the adjoint of $X$ with respect to $g_J$. Hence, we have 
     \[
     \widetilde U(\cdot,\cdot)=-U(X\cdot, \cdot)-U(\cdot, X\cdot)-UX-X^*U \ .
     \]
     By using the decomposition $X=X^s+X^a$ in its symmetric and skew-symmetric part, we can write the second component of $V_X(J,U)$ as: 
     \begin{equation}\label{decompositiontracepart}
     \underbrace{-U(X^a\cdot,\cdot)-U(X^s\cdot,\cdot)-U(\cdot,X^{a}\cdot)-U(\cdot, X^{s}\cdot)+[X^a,U]}_{\text{trace-less part}}-\underbrace{(UX^s+X^sU)}_{\text{trace part}} \ . 
     \end{equation}
     We now need to check that the function $\mu_{X}$ in the statement satisfies the properties of a moment map, as in Definition \ref{def:momentmap}. The equivariance (Property (i)) is a direct consequence of Lemma \ref{eq:SLactiononU}: if $A\in\SL(2,\R)$ and $X\in\Lsl(2,\R)$, then \begin{align*}
         \mu_{A\cdot (J,U)}(X) &=\bigg(1-f\bigg(\frac 1{2}\vl\vl A\cdot U\vl\vl_{A\cdot J}^2\bigg)\bigg)\tr(AJA^{-1}X) \\ &=\bigg(1-f\bigg(\frac 1{2}\vl\vl U\vl\vl^{2}_{J} \bigg)\bigg)\tr(JA^{-1}XA) \\ &=\mu_{(J,U)}\circ\Ad(A^{-1})(X) \\ &=\Ad^*(A)(\mu_{(J,U)})(X) \ .
     \end{align*} \\
     Let us now move to Property (ii). Let $\mu^X:\maximal\to\R$ be the map $$\mu^X(J,U)=\Big(1-f\bigg(\frac{\| U\|^2_J}{2}\bigg)\Big)\tr(JX) \ , $$ then 
     \begin{align*}
         \mathrm{d}_{(J,U)}\mu^X(\dot J,\dot U)&=-\frac {f'}{2}\bigg(\vl\vl U\vl\vl^2_J\bigg)'\tr(JX)+(1-f)\tr(\dot JX) \\ 
         &=-f'\langle U,\dot U_0\rangle\tr(JX)+(1-f)\tr(\dot JX) \ .
     \end{align*}
    Now let $V_X$ be the infinitesimal generator of the action, then 
    \begin{equation}\label{iotaomega}
    \begin{aligned}
         \iota_{V_X}\ome_f(\dot J,\dot U)&=\g_{f}(V_X(J,U), \i(\dot J,\dot U)) \\
         &=\underbrace{\frac{f-1}{2}\tr([X,J]J\dot J)}_{(i)}+\frac{f'}{2}\bigg[\underbrace{\langle [X^a, U], -\dot{U}_{0}J\rangle}_{(ii)}  \\ 
         &\ \ \ +\underbrace{\langle -U(X^{a}\cdot, \cdot)-U(\cdot, X^{a}\cdot), -\dot{U}_{0}J \rangle}_{(iii)} + \\
         &\ \ \ - \langle U(X^{s}\cdot, \cdot)-U(\cdot, X^{s}\cdot), -\dot{U}_{0}J \rangle - \langle -(UX^{s}+X^{s}U), -\dot{U}_{tr}(\cdot, J\cdot)\rangle \bigg]
     \end{aligned}\end{equation}
     where we used the decomposition in (\ref{decompositiontracepart}). The term $(i)$ is
     \begin{align*}
         \frac{1-f}{2}\tr(\dot JX+JXJ\dot J)=(1-f)\tr(\dot J X)
     \end{align*}
     by trace symmetry and $\dot J J+J\dot J=0$. Using the fact that $X^{a}=-\frac{1}{2}\trace(JX)J$ and the properties of the tensor $U$ in Lemma \ref{lem:propertiesUeT}, the sum of terms $(ii)$ and $(iii)$ can be further simplified into
     \begin{align*}
         (ii)+(iii) &= -\frac{f'}{4}\trace(JX) \bigg[ 2\langle UJ, \dot{U}_{0}J\rangle + 2\langle UJ, \dot{U}_{0}J \rangle \bigg] \\
         & = -f'\trace(JX)\langle U, \dot{U}_{0} \rangle \ .
     \end{align*}
     It only remains to show that
     \begin{equation}\label{lastequation}
      \langle U(X^s\cdot, \cdot)+U(\cdot, X^{s}\cdot),  \dot{U}_{0}J \rangle = \langle UX^s+X^sU, \dot{U}_{tr}(\cdot, \cdot) \rangle  
     \end{equation}
     To do so, we will use a basis of $\Lsl(2,\R)$, namely we can write $\Lsl(2,\R)=\Span_{\R}(\xi_1,\xi_2,\xi_3\}$ where $$\xi_1=J_0,\qquad\xi_2=\begin{pmatrix}
     1 & 0 \\ 0 &-1
     \end{pmatrix},\qquad \xi_3=\begin{pmatrix}
     0 & 1 \\ 1 & 0
     \end{pmatrix} \ . $$
     The only symmetric matrices of this basis are $\xi_2$ and $\xi_3$, hence it is sufficient to prove equation (\ref{lastequation}) when $X^s=\xi_2$ and $X^s=\xi_3$, since all the elements are linear in $X\in\Lsl(2,\R)$. In both cases we use the description in coordinates $z=x+iy$ and $w=u+iv$ for points in $\mathcal{M}(\R^{2})$ and their variations, introduced in Section \ref{sec:coordinatedescription}. In particular we can do the computation in $(z,w)=(i,w)$ by $\SL(2,\R)$-invariance. We denote by $\{e_{1}, e_{2}\}$ a $g_{J}$-orthonormal basis and by $\{e_{1}^{*}, e_{2}^{*}\}$ its dual. Moreover, in order to simplify the notation, we put $e_{ij}^{*}:=e_{i}^{*}\otimes e_{j}^{*}$.
     \begin{itemize}
         \item Consider $X^s=\xi_2$. In this case $X^s e_{1}=e_{1}$ and $X^se_{2}=-e_{2}$, hence 
         \begin{align*}
             U(X^s\cdot, \cdot)&=(e_{11}^{*}+e_{22}^{*})U_1+(e_{12}^{*}-e_{21}^{*})U_2 \\
             U(\cdot, X^{s}\cdot)&=(e_{11}^{*}+e_{22}^{*})U_1+(-e_{12}^{*}+e_{21}^{*})U_2 \ .
         \end{align*}
         Therefore, 
         \begin{align*}
             \langle U(X^{s}\cdot,\cdot)+U(\cdot, X^{s}\cdot), \dot{U}_{0}J \rangle &= \frac{1}{2} \bigg[ \trace((g_{J}^{-1}e_{11}^{*})^{2})\trace(U_{1}(\dot{U}_{1})_{0}J) + \\
             & \ \ \ +\trace((g_{J}^{-1}e_{22}^{*})^{2})\trace(U_{1}(2EJ-(\dot{U}_{1})_{0}J)) \bigg] \\
             & = \frac{1}{2} \bigg[\trace(U_{1}(\dot{U}_{1})_{0}J)+2\trace(U_{1}EJ)-\trace(U_{1}(\dot{U}_{1})_{0}J) \bigg] \\
             & = \trace(U_{1}EJ) = 2|w|^{2}\dot{x} \ .
         \end{align*}
         On the other hand,
         \begin{align*}
             UX^{s}+X^{s}U &= (e_{11}^{*}-e_{22}^{*})(U_{1}X^{s}+X^{s}U_{1})+(e_{12}^{*}+e_{21}^{*})(U_{2}X^{s}+X^{s}U_{2}) \\
             \dot{U}_{tr}(\cdot, J\cdot) &= \frac{1}{2}(e_{11}^{*}-e_{22}^{*})\trace(\dot{U}_{2})\mathds{1} -\frac{1}{2}(e_{12}^{*}+e_{21}^{*})\trace(\dot{U}_{1})\mathds{1} \ ,
         \end{align*}
         thus
         \begin{align*}
             \langle UX^{s}+X^{s}U, \dot{U}_{tr}(\cdot, J\cdot) \rangle &= \frac{1}{4}\bigg[ \frac{1}{2}\trace(\dot{U}_{2})\trace((g_{J}^{-1}(e_{11}^{*}-e_{22}^{*}))^{2})\trace(U_{1}X^{s}+X^{s}U_{1}) + \\
             & \ \ \ -\frac{1}{2}\trace(\dot{U}_{1})\trace((g_{J}^{-1}(e_{11}^{*}+e_{22}^{*}))^{2})\trace(U_{2}X^{s}+X^{s}U_{2}) \bigg] \\
             &= \frac{1}{2}[ \trace(\dot{U}_{2})\trace(U_{1}X^{s}) - \trace(\dot{U}_{1})\trace(U_{2}X^{s})] \\ 
             &=\frac{1}{2}[4(u\dot{x}-v\dot{y})u + 2(u\dot{y}+v\dot{x})v] = 2|w|^{2}\dot{x} \ . 
         \end{align*}
         \item If $X_{s}=\xi_{3}$, we have instead $X_{s}e_{1}=e_{2}$ and $X_{s}e_{2}=e_{1}$, hence
         \begin{align*}
             U(X^s\cdot, \cdot)&=(e_{11}^{*}+e_{22}^{*})U_2+(-e_{12}^{*}+e_{21}^{*})U_1 \\
             U(\cdot, X^{s}\cdot)&=(e_{11}^{*}+e_{22}^{*})U_2+(e_{12}^{*}-e_{21}^{*})U_1 \ .
         \end{align*}
         Therefore, 
         \begin{align*}
             \langle U(X^{s}\cdot,\cdot)+U(\cdot, X^{s}\cdot), \dot{U}_{0}J \rangle & = \frac{1}{2} \bigg[\trace(U_{2}(\dot{U}_{1})_{0}J)+2\trace(U_{2}EJ)-\trace(U_{2}(\dot{U}_{1})_{0}J) \bigg] \\
             & = \trace(U_{2}EJ) = -2|w|^{2}\dot{y} \ .
         \end{align*}
        On the other hand,
        \begin{align*}
             \langle UX^{s}+X^{s}U, \dot{U}_{tr}(\cdot, J\cdot) \rangle &= \frac{1}{4}\bigg[ \frac{1}{2}\trace(\dot{U}_{2})\trace((g_{J}^{-1}(e_{11}^{*}-e_{22}^{*}))^{2})\trace(U_{1}X^{s}+X^{s}U_{1}) + \\
             & \ \ \ -\frac{1}{2}\trace(\dot{U}_{1})\trace((g_{J}^{-1}(e_{11}^{*}+e_{22}^{*}))^{2})\trace(U_{2}X^{s}+X^{s}U_{2}) \bigg] \\
             &= \frac{1}{2}[ \trace(\dot{U}_{2})\trace(U_{1}X^{s}) - \trace(\dot{U}_{1})\trace(U_{2}X^{s})] \\ 
             &=2v(u\dot{x}-v\dot{y}) - 2u(u\dot{y}+v\dot{x}) = -2|w|^{2}\dot{y} \ .
        \end{align*}
     \end{itemize}
     Thus, equation (\ref{lastequation}) is proved and the theorem as well. 
\end{proof}
\subsection{Global action-angle variables}\label{sec:globalDarboux}
Let $(M,\omega)$ be a symplectic manifold and $f\in C^\infty(M,\R)$, then the \emph{Hamiltonian vector field} $\mathbb X_f\in\Gamma(TM)$ associated with $f$ is defined by the following property \begin{equation}\label{hamiltonianfield}\omega(\mathbb X_f,Y)=\mathrm df(Y), \quad\forall Y\in\Gamma(TM)\ .\end{equation}
If $(M,\omega)$ is an Hamiltonian system, namely it is a symplectic manifold endowed with a smooth function $H$, then $f\in C^\infty(M,\R)$ is called an integral of motion if $$\omega(\mathbb X_f,\mathbb X_H)=0 \ .$$ An Hamiltonian system $(M,\omega,H)$ is called \emph{completely integrable} if there exist $n$ integral of motions $f_1=H,\dots,f_n$ such that \newline $(i)$ The differentials $(\mathrm d f_1)_p,\dots,(\mathrm df_n)_p$ are linearly independent for each $p\in M$. \newline $(ii)$ They are pairwise in involution, i.e. $\omega(\mathbb X_{f_i},\mathbb X_{f_j})=0$ for each $i,j=1,\dots,n$. \newline To each completely integrable Hamiltonian system, special coordinates $\{\theta_1,\psi_1,\dots,\theta_n,\psi_n\}$ on $M$ can be found as a consequence of the Arnlod-Liouville Theorem (\cite{arnol2013mathematical}). They form a global Darboux frame for the symplectic form $$\omega=\sum_{i=1}^n\mathrm d\theta_i\wedge\mathrm d\psi_i \ ,$$ but in general, the action variables $\{\psi_1,\dots,\psi_n\}$ are not the given integral of motions. In the following we will show first that our space $\maximalflat$ can be endowed with the above structure and later that, two of the four coordinates given by the Arnlod-Liouville Theorem are actually the integral of motions, and thus have geometric significance. \\ \\ In the following, we will use the existence of the Hamiltonian actions of $S^1$ and $\SL(2,\R)$ on $\maximalflat$ (see Section \ref{sec:circleaction} and Section \ref{sec:SLaction}). 
    \begin{lemma}\label{lem:restrictedhamiltonian}
Let $(M,\omega)$ be a symplectic manifold endowed with a Hamiltonian $G$-action and moment map $\mu_G:M\to\mathfrak g^*$. If $H\le G$ is any closed subgroup, then the restricted $H$-action is Hamiltonian with moment map $\mu_H:M\to\mathfrak h^*$ given by $\mu_H:=|_{\mathfrak h}\circ\mu_G$, where $|_{\mathfrak h}:\mathfrak g^*\to\mathfrak h^*$ is the map which associates to each functional on $\mathfrak g$ its restriction on $\mathfrak h$.
\end{lemma}
Notice that inside $\SL(2,\R)$ there is the subgroup of diagonal matrices with determinant equal to one, namely
\begin{equation}
    \bigg\{\begin{pmatrix}\lambda & 0 \\ 0 & \frac{1}{\lambda} \end{pmatrix} \ \bigg| \ \lambda\in\R^*\bigg\}<\SL(2,\R) \ .
\end{equation}In particular, such a subgroup can be identified with a copy of $\R^*$ which still acts in a Hamiltonian fashion, by the previous lemma, on the space $\maximalflat$. Moreover, from Theorem \ref{thm:flatmaximalandquartic} we know that $\maximalflat$ can be identified with $\Hyp\times\C^*$. Let $z=x+iy$ be a complex coordinate on $\Hyp$ and $w=u+iv$ on $\C$, then the Hamiltonian function of the circular action is given by $H_1(z,w)=\frac{1}{2}f\big(y^4|w|^2\big)$. \begin{lemma}
    Let $\R^*$ be a copy of the subgroup of diagonal matrices in $\SL(2,\R)$ and consider its restricted Hamiltonian action on $\maximalflat$, then the Hamiltonian function is given by $$H_2(z,w)=2\frac{x}{y}\big(1-f(y^4|w|^2)\big) \ .$$
\end{lemma}
\begin{proof}
The Lie algebra of $\R^*$ can be identified with $$\mathfrak h:=\bigg\{\begin{pmatrix}\alpha & 0 \\ 0 & -\alpha 
\end{pmatrix} \ \bigg| \ \alpha\in\R\bigg\} \ .$$  If $\mu$ denotes the moment map for the $\SL(2,\R)$-action of Proposition \ref{prop:momentmapSL2R}, then the associated moment map for the restricted $\R^*$-action $\mu_\mathfrak h$ is $$\mu^X_\mathfrak h(z,w) =\big(1-f(y^4|w|^2)\big)\tr(j(z)X),$$ where $X\in\mathfrak h$. Let $\xi:=\begin{pmatrix}1 & 0 \\ 0 & -1
\end{pmatrix}\in\mathfrak h$, then the Hamiltonian function $H_2:\maximalflat\to\R$ is $H_2(z,w):=\mu^\xi_\mathfrak h(z,w)$, given that $\mathrm d\mu^\xi_\mathfrak h=\ome_f(V_\xi,\cdot)$, where $$ V_\xi=2\bigg(x\frac{\partial}{\partial x}+y\frac{\partial}{\partial y}\bigg)-4\bigg(u\frac{\partial}{\partial u}+v\frac{\partial}{\partial v}\bigg)$$ is the infinitesimal generator of the action. Finally, since $$\tr(j(z)\xi)=\tr\bigg(\begin{pmatrix}\frac{x}{y} & -\frac{x^2+y^2}{y} \\ \frac 1{y} & -\frac x{y}
\end{pmatrix}\cdot\begin{pmatrix} 1 & 0 \\ 0 & -1
\end{pmatrix}\bigg)=2\frac{x}{y}$$ we get $\displaystyle H_2(z,w)=2\frac{x}{y}\big(1-f(y^4|w|^2)\big)$. 
\end{proof}
\begin{proposition}\label{prop:deftintegrablesystem}
The Hamiltonian system $(\maximalflat,\ome_f,H_1)$ is completely integrable. The integral of motions are given by $$H_1(z,w)=\frac{1}{2}f\big(y^4|w|^2\big),\quad H_2(z,w)=2\frac{x}{y}\big(1-f(y^4|w|^2)\big) \ .$$
\end{proposition}
\begin{proof}
Let $\Xuno,\Xdue$ be the Hamiltonian vector fields associated with $H_1,H_2$. An explicit expression is given by $$\Xuno=u\frac{\partial}{\partial v}-v\frac{\partial}{\partial u},\quad \Xdue=2\bigg(x\frac{\partial}{\partial x}+y\frac{\partial}{\partial y}\bigg)-4\bigg(u\frac{\partial}{\partial u}+v\frac{\partial}{\partial v}\bigg) \ .$$ It is clear that they are point-wise linearly independent on $\maximalflat$, hence to end the proof we only need to show that they are involutive. The symplectic form is \begin{align*}(\ome_f)_{(z,w)}=&\bigg(-1+f-4f'y^4(u^2+v^2)\bigg)\frac{\dx\wedge\dy}{y^2}-f'y^4\du\wedge\devu \\& \ \ \ \ -2y^3f'\bigg(u(\dx\wedge\du+\dy\wedge\devu)+v(\du\wedge\dy-\devu\wedge\dx)\bigg) \ .\end{align*}Then, \begin{align*}
    \ome_f(\Xuno,\Xdue)=2&
    \bigg(ux\ome_f\bigg(\partialv,\partialx\bigg)+uy\ome_f\bigg(\partialv,\partialy\bigg)-vx\ome_f\bigg(\partialu,\partialx\bigg)- \\ &vy\ome_f\bigg(\partialu,\partialy\bigg)\bigg)-4\bigg(u^2\ome_f\bigg(\partialv,\partialu\bigg)-v^2\ome_f\bigg(\partialu,\partialv\bigg)\bigg) \ .
\end{align*}By using the explicit expression of the symplectic form, it is easy to see that the right hand side of the above equality is zero.
\end{proof}
The map $\pi:=(H_1,H_2):\maximalflat\to B$ is a Lagrangian fibration over the open subset $B:=\pi\big(\maximalflat\big)\subset\R^2$, which is actually contained in the set of regular values of $\pi$. In particular, by using the explicit description of the Hamiltonian functions, it is not difficult to see that $B$ is diffeomorphic to $\{(u_1,u_2)\in \R^2 \ | \ u_1<0\}$, hence it is contractible. The fibre over a point $b\in B$ is $$\pi^{-1}(b)=\big\{(z,w)\in\maximalflat \ | \ \frac{1}{2}f\big(y^4|w|^2\big)=b_1, \ 2\frac{x}{y}\big(1-f(y^4|w|^2)\big)=b_2\big\} \ ,$$ which is a copy of $\R\times S^1$. Let us denote with $\{\theta_1,\theta_2\}$ the angle coordinates given by the Arnold-Liouville theorem applied to $\big(\maximalflat,\ome_f, H_1, H_2\big)$, then we can prove the following
\begin{theorem}
The coordinates $\{\theta_1,H_1,\theta_2,H_2\}$ are a global Darboux frame for $\ome_f$, or in other words $$\ome_f=\mathrm d\theta_1\wedge\mathrm d H_1+\mathrm d\theta_2\wedge\mathrm dH_2 \ .$$
\end{theorem}
\begin{proof}
The fibre $\pi^{-1}(b)$ over a point $b\in B$ is a Lagrangian submanifold of $(\maximalflat,\ome_f)$ diffeomorphic to $S^1\times\R$. Moreover, the Hamiltonian vector fields $\Xuno,\Xdue$ are complete on $\pi^{-1}(b)$, indeed $\Xuno$ is the generator of the counter clock-wise rotation in the plane and the integral curve of $\Xdue$ passing through the point $(z,w)$ is $\gamma_{(z,w)}(t)=(e^{2t}z,e^{-3t}w)$, which is defined for all $t\in\R$. This implies that $\pi:\maximalflat\to B$ is a complete Lagrangian fibration (\cite[Definition 3.7]{rungi2022global}) which fibres over a contractible open subset of $\R^2$. By \cite[Corollary 3.23]{rungi2022global} there exists a global Lagrangian section $\sigma:B\to\maximalflat$ which allows us to conclude the proof following the same strategy as \cite[Theorem A]{rungi2022global}.
\end{proof}

\printbibliography

\end{document}